\documentclass{amsart}

\usepackage{amsmath,amsthm,amsfonts,amssymb,amscd,latexsym}

\newtheorem{thm}{Theorem}[section]
\newtheorem{lem}[thm]{Lemma}
\newtheorem{pro}[thm]{Proposition}
\newtheorem{cor}[thm]{Corollary}

\numberwithin{equation}{section}

\newcommand{\id}{\mathrm{id}}
\newcommand{\op}{\mathrm{op}}
\newcommand{\ad}{\mathrm{ad}}
\newcommand{\tr}{\mathrm{tr}}
\newcommand{\rad}{\mathrm{Rad}}

\newcommand{\der}{\mathrm{Der}}
\newcommand{\ider}{\mathrm{Ider}}
\newcommand{\out}{\mathrm{Out}}
\newcommand{\lie}{\mathrm{Lie}}
\newcommand{\leib}{\mathrm{Leib}}
\newcommand{\sym}{\mathrm{sym}}
\newcommand{\bi}{\mathrm{bi}}
\newcommand{\Ker}{\mathrm{Ker}}
\newcommand{\im}{\mathrm{Im}}
\newcommand{\mult}{\mathrm{Mult}}
\newcommand{\ann}{\mathrm{Ann}}
\newcommand{\End}{\mathrm{End}}
\newcommand{\Hom}{\mathrm{Hom}}

\newcommand{\HL}{\mathrm{HL}}
\newcommand{\HCE}{\mathrm{H}}
\newcommand{\CL}{\mathrm{C}}

\newcommand{\N}{\mathbb{N}}
\newcommand{\F}{\mathbb{F}}
\newcommand{\C}{\mathbb{C}}

\newcommand{\af}{\mathfrak{A}}
\newcommand{\rf}{\mathfrak{R}}
\newcommand{\Sf}{\mathfrak{S}}
\newcommand{\ssf}{\mathfrak{s}}
\newcommand{\tf}{\mathfrak{T}}
\newcommand{\Bf}{\mathfrak{B}}
\newcommand{\lf}{\mathfrak{L}}
\newcommand{\If}{\mathfrak{I}}
\newcommand{\IIf}{\mathfrak{i}}
\newcommand{\jf}{\mathfrak{J}}
\newcommand{\kf}{\mathfrak{K}}
\newcommand{\nf}{\mathfrak{N}}
\newcommand{\ef}{\mathfrak{E}}

\newcommand{\slf}{\mathfrak{sl}}
\newcommand{\gl}{\mathfrak{gl}}

\begin{document}


\title[Leibniz algebras as non-associative algebras]{Leibniz algebras as non-associative algebras}

\author{J\"org Feldvoss}
\address{Department of Mathematics and Statistics, University of South Alabama,
Mobile, AL 36688-0002, USA}
\email{jfeldvoss@southalabama.edu}




\begin{abstract}
In this paper we define the basic concepts for left or right Leibniz algebras and prove
some of the main results. Our proofs are often variations of the known proofs but
several results seem to be new.
\end{abstract}


\date{October 17, 2018}
          
\maketitle


\section*{Introduction}


This paper is mainly a survey on (left or right) Leibniz algebras from the point of view of
the theory of non-associative (i.e., not necessarily associative) algebras, but there are
also several new results. Leibniz algebras were first introduced by Bloh in the mid sixties
of the last century (see \cite{Bl1,Bl2,Bl3}) and then forgotten for nearly thirty years. In
the early 1990's  they were rediscovered by Loday who together with his students and
collaborators developed much of the theory of Leibniz algebras, Leibniz bimodules, and
Leibniz cohomology (see, for example, \cite{L,LP1,C}). A left (resp.\ right) Leibniz algebra
is a vector space with a multiplication for which every left (resp.\ right) multiplication operator
is a derivation (i.e., a linear operator satisfying the usual Leibniz product rule). As such
Leibniz algebras are non-anticommutative versions of Lie algebras. In particular, Leibniz
algebras are examples of non-associative algebras (see \cite{S}). In contrast to other
papers on this topic, we study Leibniz algebras exclusively from this point of view. We
have tried to make the paper sufficiently self-contained so that it could serve as a first
introduction to Leibniz algebras, their modules (or representations), and their cohomology.
Leibniz algebras play an important role in different areas of mathematics and physics (see
\cite{L}). In the last three decades numerous papers on Leibniz algebras appeared and
many results have been duplicated. In this paper we develop the basics of the theory of
Leibniz algebras in a systematic way by considering them as a special class of non-associative
algebras. In the following we will describe the contents of the paper in more detail.

The first section is devoted to some background material on non-associative algebras which
will be useful for the rest of the paper. In particular, we introduce the concept of a radical
of an arbitrary algebra and prove three of its properties that can be used as axioms for
such a concept. We refer the reader to \cite{S} for more details and most of the proofs.

In the second section we give the definition of left resp.\ right Leibniz algebras and prove
some basic results. We use several low-dimensional Leibniz algebras to illustrate the concepts
to be introduced and the results to be proved. Among other notions, we define the (left/right)
center and the Leibniz kernel of a Leibniz algebra. The latter measures how much a Leibniz
algebra deviates from being a Lie algebra. Moreover, we associate several Lie algebras to
a Leibniz algebra and discuss how these are related to each other. We include an example
showing that left or right Leibniz algebras are not necessarily power-associative. This falsifies
a claim by Barnes in \cite{B1}. On the other hand, we prove that symmetric Leibniz algebras
(i.e., algebras satisfying the left and the right Leibniz identity) are flexible, power-associative,
and nil. The terms of the derived series of an arbitrary algebra are usually only subalgebras.
At the end of the second section we show that for left or right Leibniz algebras each term
of their derived series is an ideal.

Section 3 contains definitions of left Leibniz modules and Leibniz bimodules of a left Leibniz
algebra. In particular, following Eilenberg \cite{E} we motivate the defining identities of
a Leibniz bimodule by considering abelian extensions of a left Leibniz algebra. We prove
some basic properties of Leibniz bimodules following mainly Loday \cite{L,LP1} who
introduced and investigated Leibniz bimodules for a right Leibniz algebra. In addition, we
also briefly discuss trace forms associated to finite-dimensional left Leibniz modules (see
also \cite{AAO2} and \cite{DMS1}). Similarly to the Leibniz kernel of a Leibniz algebra, we
introduce the anti-symmetric kernel of a Leibniz bimodule. Using this concept, we give a
very short ``Schur's lemma type" proof of the fact that irreducible Leibniz bimodules are
either symmetric or anti-symmetric (see also the proof of Theorem~3.1 in \cite{FM}). Note
that this proof neither needs to assume that the Leibniz algebra is finite dimensional nor
that the Leibniz bimodule is finite dimensional as in \cite[Theorem 1.4]{B2}). In the fourth
section we define the cohomology of a left Leibniz algebra in analogy to the cohomology
of a right Leibniz algebra in \cite{L,LP1}, and describe the Leibniz cohomology spaces in
degree 0 and 1. We also show explicitly how Leibniz $2$-cocycles give rise to abelian
extensions of left Leibniz algebras.

The remaining three sections of the paper are devoted to several results for nilpotent,
solvable, and semisimple Leibniz algebras, respectively. We give variants of the known
proofs of Engel's and Lie's theorem for Leibniz algebras as well as derive some of their
applications. Furthermore, we prove Cartan's solvability criterion for Leibniz algebras.
We characterize the nilpotency and solvability of a Leibniz algebra in terms of the nilpotency
and solvability of their associated Lie algebras, respectively. These results seem to be
new. In the last section we derive some structural properties of (semi)simple Leibniz
algebras and explain that the first Whitehead lemma does not hold for Leibniz algebras.
In a previous version of this paper we derived the second Whitehead lemma for
Leibniz algebras from Levi's theorem for Leibniz algebras along the lines of the proof
of \cite[Proposition 3.22]{S}. Unfortunately, our proof was not correct. We are very
grateful to Bakhrom Omirov for bringing this to our attention.\footnote{Recently, in
joint work with Friedrich Wagemann we found a proof of the second Whitehead lemma
for Leibniz algebras by using spectral sequences.}

In this paper all algebras are defined over a field. For a subset $X$ of a vector space
$V$ over a field $\F$ we let $\langle X\rangle_\F$ be the subspace of $V$ that is spanned
by $X$. We use $[-,-]$ to denote the commutator of linear operators or matrices.
The identity function on a set $S$ will be denoted by $\id_S$, the set of non-negative
integers will be denoted by $\N_0$, and the set of positive integers will be denoted by
$\N$.


\section{Non-associative algebras}


In this section we briefly recall some of the definitions and results on non-associative
algebras that we will need in the remainder of the paper. For more details and most of the
proofs we refer the reader to \cite{S}.

An {\em algebra\/} $\af$ is a vector space over a field with a bilinear mapping $\af\times
\af\to\af$, $(x,y)\mapsto xy$, the {\em multiplication\/} of $\af$. The usual definitions of
the concepts of {\em subalgebra\/}, {\em left\/} or {\em right ideal\/}, {\em ideal\/} (=
left and right ideal), {\em homomorphism\/}, {\em isomorphism\/}, etc., are the same as
for associative algebras since they do not use the associativity of the multiplication. Moreover,
the {\em fundamental homomorphism theorem\/}, the {\em isomorphism theorems\/}, and
the {\em correspondence theorems for subalgebras\/} and for {\em one-sided\/} or {\em
two-sided ideals\/}, respectively, continue to hold with the same proofs as in the associative
case.

Let $\Sf$ and $\tf$ be two non-empty subsets of an algebra $\af$ over a field $\F$. Then
$$\Sf\tf:=\langle st\mid s\in\Sf,t\in\tf\rangle_\F$$ is the $\F$-subspace of $\af$ spanned
by the products $st$. In particular, $\Sf^2:=\Sf\Sf$ (see \cite[p.\ 9]{S}).

The {\em derived series\/} of subalgebras $$\af^{(0)}\supseteq\af^{(1)}\supseteq\af^{(2)}
\supseteq\cdots$$ of $\af$ is defined recursively by $\af^{(0)}:=\af$ and $\af^{(n+1)}:=
(\af^{(n)})^2$ for every non-negative integer $n$.

Note that $\af^{(m+n)}=(\af^{(m)})^{(n)}$ for all non-negative integers $m$ and $n$.
Moreover, if $\phi:\af\to\Bf$ is a homomorphism of algebras, then $\phi(\af^{(n)})=\phi
(\af)^{(n)}$ for every non-negative integer $n$.

An algebra $\af$ is called {\em solvable\/} if $\af^{(r)}=0$ for some non-negative integer
$r$. An algebra $\af$ is called {\em abelian\/} if $\af\af=0$ (i.e., if any product of elements
in $\af$ is zero).

The next result is an immediate consequence of the compatibility of homomorphisms of algebras
with the derived series.

\begin{pro}\label{subalghomimsolv}
Subalgebras and homomorphic images of solvable algebras are solvable.
\end{pro}

The following results are well-known (see \cite[Proposition 2.2 and 2.3]{S}).

\begin{pro}\label{extsolv}
Extensions of solvable algebras by solvable algebras are solvable.
\end{pro}

\begin{pro}\label{sumidsolv}
The sum of two solvable ideals of an algebra is solvable.
\end{pro}

Let $\af$ be a finite-dimensional algebra, and let $\rf$ be a solvable ideal of maximal dimension.
If $\If$ is any solvable ideal of $\af$, then it follows from Proposition \ref{sumidsolv} that $\rf+
\If$ is a solvable ideal of $\af$. Since $\rf\subseteq\rf+\If$, for dimension reasons we have that
$\rf=\rf+\If$, and thus $\If\subseteq\rf$. This shows that $\rf$ is the largest solvable ideal of
$\af$. The ideal $\rf$ is called the {\em radical\/} of the algebra $\af$ and will be denoted by
$\rad(\af)$.

\begin{pro}\label{rad}
Every finite-dimensional algebra $\af$ contains a largest solvable ideal $\rad(\af)$ satisfying the
following properties:
\begin{enumerate}
\item[(a)] $\rad(\rad(\af))=\rad(\af)$.
\item[(b)] $\rad(\af/\rad(\af))=0$.
\item[(c)] If $\phi:\af\to\Bf$ is a homomorphism of algebras, then $\phi(\rad(\af))\subseteq\rad
                (\phi(\af))$.
\end{enumerate}
\end{pro}

\begin{proof}
(a) follows from the solvability of $\rad(\af)$, and (c) is an immediate consequence of Proposition
\ref{subalghomimsolv}.

(b): According to the correspondence theorem for ideals, there exists an ideal $\If$ of $\af$ such
that $\rad(\af)\subseteq\If$ and $\rad(\af/\rad(\af))=\If/\rad(\af)$. This implies that $\If/\rad(\af)$
is solvable, and therefore Proposition \ref{extsolv} yields that $\If$ is solvable. Hence $\If\subseteq
\rad(\af)$, and so $\rad(\af/\rad(\af))=0$.
\end{proof}

Let $a\in\af$ be an arbitrary element. Then the left multiplication operator $L_a:\af\to\af$, $x
\mapsto ax$ is linear and $$L(\af):=\{L_a\mid a\in\af\}$$ is a subspace of the associative algebra
$\End(\af)$ of linear operators on $\af$. Similarly, the right multiplication operator $R_a:\af\to
\af$, $x\mapsto xa$ is linear and $$R(\af):=\{R_a\mid a\in\af\}$$ is a subspace of $\End(\af)$.
Let $\mult(\af)$ denote the subalgebra of $\End(\af)$ that is generated by $L(\af)\cup R(\af)$,
the {\em associative multiplication algebra\/} of $\af$ (see \cite[Section 2 in Chapter II]{S}). Let
$\gl(\af)$ denote the {\em general linear Lie algebra\/} on the underlying vector
space of $\af$ with the commutator $[X,Y]:=X\circ Y-Y\circ X$ as Lie bracket. Moreover, let $\lie
(\af)$ denote the subalgebra of $\gl(\af)$ that is generated by $L(\af) \cup R(\af)$, the {\em
Lie multiplication algebra\/} of $\af$ (see \cite[Section 3 in Chapter II]{S}). Finally, let $$\der
(\af):=\{D\in\End (\af)\mid\forall\,x,y\in\af:D(xy)=D(x)y+xD(y)\}$$ denote the {\em derivation
algebra\/} of $\af$. Note that $\der(\af)$ is a subalgebra of $\gl(\af)$, and therefore $\der(\af)$
is another Lie algebra associated to $\af$.

An algebra $\af$ is called {\em nilpotent\/} if there exists a positive integer $n$ such that any
product of $n$ elements in $\af$, no matter how associated, is zero. This generalizes the concept
of nilpotency for associative algebras. Note that every nilpotent algebra is solvable (see \cite[p.
18]{S}). For any subset $\Sf$ of an algebra $\af$ let $\Sf^*$ denote the subalgebra of $\mult
(\af)$ generated by $\{L_s:\af\to\af\mid s\in\Sf\}\cup\{R_s:\af\to\af\mid s\in\Sf\}$. Then an
ideal $\If$ of $\af$ is nilpotent if, and only if, $\If^*$ is nilpotent (see \cite[Theorem~2.4]{S}).
In particular, $\af$ is nilpotent if, and only if, $\mult(\af)=\af^*$ is nilpotent.

An algebra $\af$ is called {\em flexible\/} if $L_x\circ R_x=R_x\circ L_x$ holds for every element
$x\in\af$, or equivalently, if the identity $x(yx) =(xy)x$ is satisfied for all elements $x,y\in\af$
(see \cite[p.\ 28]{S}). An algebra $\af$ is called {\em power-associative\/} if any subalgebra of
$\af$ generated by one element is associative (see \cite[p.\ 30]{S}). In this case one can define
powers of an element $x\in\af$ recursively by $x^1:=x$ and $x^{n+1}:=xx^n$ for every positive
integer $n$. These powers then satisfy the usual power laws $x^{m+n}=x^mx^n$ and $(x^m)^n
=x^{mn}$ (see \cite[p.\ 30]{S}). Alternative algebras, Jordan algebras, and Lie algebras are flexible
and power-associative (see \cite[pp. 28, 30, and 92]{S}). 

An element $x$ of a power-associative algebra is called {\em nilpotent\/} if $x^n=0$ for some
positive integer $n$. A subset of a power-associative algebra consisting only of nilpotent elements
is called {\em nil\/} (see \cite[p.\ 30]{S}). Note that every solvable power-associative algebra is
nil (see \cite[p.\ 31]{S}).

The {\em opposite algebra\/} $\af^\op$ of an algebra $\af$ with multiplication $(x,y)\mapsto xy$
has the same underlying vector space structure and the opposite multiplication $(x,y)\mapsto x
\cdot y:=yx$. Since the derived series of $\af^\op$ coincides with the derived series of $\af$,
the opposite algebra of a solvable algebra is solvable. Moreover, it is clear from the definition of
nilpotency, that the opposite algebra of a nilpotent algebra is nilpotent.


\section{Leibniz algebras -- Definition and Examples}


A {\em left Leibniz algebra\/} is an algebra $\lf$ such that every left multiplication operator
$L_x:\lf\to\lf$, $y\mapsto xy$ is a derivation. This is equivalent to the identity
\begin{equation}\label{LLI}
x(yz)=(xy)z+y(xz)
\end{equation}
for all $x,y,z\in\lf$, the {\em left Leibniz identity\/}, which in turn is equivalent to the identity
\begin{equation}\label{RLLI}
(xy)z=x(yz)-y(xz)
\end{equation}
for all $x,y,z\in\lf$.

Similarly, one defines a {\em right Leibniz algebra\/} to be an algebra $\lf$ such that every right
multiplication operator $R_y:\lf\to\lf$, $x\mapsto xy$ is a derivation. This is equivalent to the
identity
\begin{equation}\label{RLI}
(xy)z=(xz)y+x(yz)
\end{equation}
for all $x,y,z\in\lf$, the {\em right Leibniz identity\/}, which in turn is equivalent to the identity
\begin{equation}\label{LRLI}
x(yz)=(xy)z-(xz)y
\end{equation}
for all $x,y,z\in\lf$.

Following Mason and Yamskulna \cite{MY} we call an algebra a {\em symmetric Leibniz algebra\/}
if it is at the same time a left and a right Leibniz algebra. Note that every Lie algebra is a symmetric
Leibniz algebra.

It is clear that the opposite algebra of a left Leibniz algebra is a right Leibniz algebra and that the
opposite algebra of a right Leibniz algebra is a left Leibniz algebra. Consequently, the opposite
algebra of a symmetric Leibniz algebra is again a symmetric Leibniz algebra. Therefore, in most
situations it is enough to consider only left or right Leibniz algebras.

The following results are direct consequences of the left and right Leibniz identity, respectively.

\begin{lem}\label{leftmultsq}
If $\lf$ is a left Leibniz algebra, then $L_{x^2}=0$ for every element $x\in\lf$.
\end{lem}

\begin{lem}\label{rightmultsq}
If $\lf$ is a right Leibniz algebra, then $R_{x^2}=0$ for every element $x\in\lf$.
\end{lem}

\begin{proof}
We only prove Lemma \ref{leftmultsq} as this yields Lemma \ref{rightmultsq} by considering
the opposite algebra. Let $x,y\in\lf$ be arbitrary elements. Then we obtain from identity
(\ref{RLLI}) that $$L_{x^2}(y)=x^2y=x(xy)-x(xy)=0\,,$$ which shows that $L_{x^2}=0$.
\end{proof}

Every abelian (left or right) Leibniz algebra is a Lie algebra, but there are many Leibniz algebras
that are not Lie algebras (see, for example, \cite{C,AO1,AO2,DA,AAO1,AAO2,AOR,ORT,CK,G,
GVKO,KRSH,CGVO,CCO,CGVOK,DMS1,D,DMS2,DMS3,DMS4}). We will use the following three
examples to illustrate the concepts introduced in this section.
\vspace{.2cm}

\noindent {\bf Examples.}
\begin{enumerate}
\item[(1)] Let $\af_\ell:=\F e\oplus\F f$ be a two-dimensional vector space with multiplication
                $ee=fe=ff=0$, and $ef=f$. Then $\af_\ell$ is a left Leibniz algebra, but not a right
                Leibniz algebra. We have that $\af_\ell^{(1)}=\F f$ and $\af_\ell^{(2)}=0$. Hence
                $\af_\ell$ is solvable.
\item[(2)] Let $\nf:=\F e\oplus\F f$ be a two-dimensional vector space with multiplication $ee=
                ef=fe=0$, and $ff=e$. Then $\nf$ is a symmetric Leibniz algebra. We will see
                in Section 5 that $\nf$ is nilpotent.
\item[(3)] Let $\Sf_\ell:=\slf_2(\C)\times L(1)$, where $\slf_2(\C)$ is the Lie algebra of traceless
                complex $2\times 2$ matrices and $L(1)$ is the two-dimensional left $\slf_2(\C)$-module
                (see \cite[Lemma 7.2]{H}). Then $\Sf_\ell$ with multiplication $(X,a)(Y,b):=([X,Y],X
                \cdot b)$ for any $X,Y\in\slf_2(\C)$ and any $a,b\in L(1)$ is a left Leibniz algebra (see
                Section~3 and Lemma \ref{ext})\footnote{This is known as the {\em hemi-semidirect product\/}
                of $\slf_2(\C)$ and $L(1)$ (see \cite[Definition 1.5]{OW}).}. Moreover, it can be shown
                that $\Sf_\ell$ is simple (see Section 7 for the definition of the simplicity of Leibniz algebras).
\end{enumerate}
\vspace{.3cm}

\noindent {\bf Remark.} One can prove that up to isomorphism $\af_\ell$, $\af_\ell^\op$, and $\nf$
are the only two-dimensional left or right non-Lie Leibniz algebras (see \cite[pp. 11/12]{DMS1}).
\vspace{.2cm}

Let $\lf$ be a left or right Leibniz algebra. Then $$C_\ell(\lf):=\{c\in\lf\mid L_c=0\}=\{c\in\lf
\mid\forall\,x\in\lf:cx=0\}$$  is called the {\it left center\/} of $\lf$, $$C_r(\lf):=\{c\in\lf
\mid R_c=0\}=\{c\in\lf\mid\forall\,x\in\lf:xc=0\}$$ is called the {\it right center\/} of $\lf$,
and $$C(\lf):=C_\ell(\lf)\cap C_r(\lf)$$ is called the {\em center\/} of $\lf$.\footnote{Note
that, as for Lie algebras, the given definition of the center of a left or right Leibniz algebra is
not the one used for other non-associative algebras (see \cite[p.\ 14]{S}). The reason for
this is that, in general, Leibniz algebras, contrary to alternative algebras or Jordan algebras,
are far from being associative (see Proposition \ref{leftass} and Proposition \ref{rightass}).}
\vspace{.4cm}

\noindent {\bf Remark.} For a Lie algebra the left center, the right center, and the center are
all the same.
\vspace{.2cm}

It is an immediate consequence of the definitions that the left center is a right ideal and the
right center is a left ideal. More precisely, we have the following results.

\begin{pro}\label{leftcen}
Let $\lf$ be a left Leibniz algebra. Then $\lf C_\ell(\lf)\subseteq C_\ell(\lf)$ and $C_\ell(\lf)\lf
=0$. In particular, $C_\ell(\lf)$ is an abelian ideal of $\lf$.
\end{pro}

\begin{pro}\label{rightcen}
Let $\lf$ be a right Leibniz algebra. Then $C_r(\lf)\lf\subseteq C_r(\lf)$ and $\lf C_r(\lf)=0$.
In particular, $C_r(\lf)$ is an abelian ideal of $\lf$.
\end{pro}

\begin{proof}
We only prove Proposition \ref{leftcen} as this yields Proposition \ref{rightcen} by considering
the opposite algebra. Let $c\in C_\ell(\lf)$ and $x,y\in\lf$ be arbitrary elements. Then we obtain
from identity (\ref{RLLI}) that $$L_{xc}(y)=(xc)y=x(cy)-c(xy)=xL_c(y)-L_c(xy)=0\,,$$ which
shows that $L_{xc}=0$, i.e., $xc\in C_\ell(\lf)$. This proves the first statement, and the second
statement is an immediate consequence of the definition of the left center.
\end{proof}

\noindent {\bf Examples.}
\begin{enumerate}
\item[(1)] For the two-dimensional solvable left Leibniz algebra $\af_\ell$ from Example 1 we
                have that $C_\ell(\af_\ell)=\F f$ and $C_r(\af_\ell)=\F e$. Hence $C(\af_\ell)=0$.
\item[(2)] For the two-dimensional nilpotent symmetric Leibniz algebra $\nf$ from Example 2
                we have that $C(\nf)=C_\ell(\nf)=C_r(\nf)=\F e$.
\item[(3)] For the five-dimensional simple left Leibniz algebra $\Sf_\ell$ from Example 3 we have
                that $C_\ell(\Sf_\ell)=L(1)$ and $C_r(\Sf_\ell)=0$. Hence $C(\Sf_\ell)=0$.
\end{enumerate}
\vspace{.2cm}

Let $\lf$ be a left or right Leibniz algebra over a field $\F$. Then $$\leib(\lf):=\langle x^2\mid x
\in\lf\rangle_\mathbb{F}$$ is called the {\em Leibniz kernel\/} of $\lf$. The Leibniz kernel measures
how much a left or right Leibniz algebra deviates from being a Lie algebra. In particular, a left or
right Leibniz algebra is a Lie algebra if, and only if, its Leibniz kernel vanishes.

The following results are well-known (see \cite[Proposition 1\,(a)]{C}, \cite[Proposition~1.1.4]{Co},
and \cite[Lemma 1.3 and its proof]{B2}).

\begin{pro}\label{leftker}
Let $\lf$ be a left Leibniz algebra. Then $\lf\leib(\lf)\subseteq\leib(\lf)$ and $\leib(\lf)\subseteq
C_\ell(\lf)$.
\end{pro}

\begin{pro}\label{rightker}
Let $\lf$ be a right Leibniz algebra. Then $\leib(\lf)\lf\subseteq\leib(\lf)$ and $\leib(\lf)\subseteq
C_r(\lf)$.
\end{pro}

\begin{cor}\label{kercen}
If $\lf$ is a symmetric Leibniz algebra, then $\leib(\lf)\subseteq C(\lf)$.
\end{cor}

\begin{proof}
We only prove Proposition \ref{rightker} as this yields Proposition \ref{leftker} by considering the
opposite algebra. As $C_r(\lf)$ is a subspace of $\lf$, the second statement is an immediate
consequence of Lemma \ref{rightmultsq}.

Since $\leib(\lf)$ is a subspace of $\lf$, it is enough for the proof of the first statement to show
that $x^2y\in\leib(\lf)$ for any elements $x,y\in\lf$. By using again Lemma \ref{rightmultsq},
we obtain that $$(x^2+y)^2=(x^2+y)(x^2+y)=x^2x^2+x^2y+yx^2+y^2=x^2y+y^2\,,$$
and therefore $x^2y=(x^2+y)^2-y^2\in\leib(\lf)$.
\end{proof}

\noindent {\bf Examples.}
\begin{enumerate}
\item[(1)] For the two-dimensional solvable left Leibniz algebra $\af_\ell$ from Example 1 we
                have that $\leib(\af_\ell)=\F f=C_\ell(\af_\ell)$. Moreover, $\leib(\af_\ell)\cap C_r
                (\af_\ell)=0$.
\item[(2)] For the two-dimensional nilpotent symmetric Leibniz algebra $\nf$ from Example 2
                we have that $\leib(\nf)=\F e=C_\ell(\nf)=C_r(\nf)=C(\nf)$.
\item[(3)] For the five-dimensional simple left Leibniz algebra $\Sf_\ell$ from Example 3 we
                have that $\leib(\Sf_\ell)=L(1)=C_\ell(\Sf_\ell)$. Moreover, $\leib(\Sf_\ell)\cap C_r
                (\Sf_\ell)=0$.
\end{enumerate}
\vspace{.2cm}

\noindent {\bf Remark.} Let $\lf$ be the two-dimensional solvable left Leibniz algebra from
Example~1 or the five-dimensional simple left Leibniz algebra from Example 3. We have
that $\lf\leib(\lf)=\leib(\lf)$. This shows that for a left Leibniz algebra $\lf$, $\lf\leib(\lf)$
can be non-zero and $\lf\leib(\lf)$ is not necessarily properly contained in $\leib(\lf)$. Of
course, a similar statement holds for right Leibniz algebras. But for a symmetric Leibniz algebra
$\lf$ we have that $\lf\leib(\lf)=0=\leib(\lf)\lf$ as $\leib(\lf)\subseteq C(\lf)$.
\vspace{.3cm}

The following result is a consequence of Proposition \ref{leftker} or Proposition \ref{rightker}
(see \cite[p.\ 11]{DMS1} and \cite[p.\ 479]{FM}).

\begin{pro}\label{ker}
Let $\lf$ be a left or right Leibniz algebra. Then $\leib(\lf)$ is an abelian ideal of $\lf$. Moreover,
if $\lf\ne 0$, then $\leib(\lf)\ne\lf$.
\end{pro}

\begin{proof}
It follows from Proposition \ref{leftker} or Proposition \ref{rightker} that in either case $\leib(\lf)$
is an ideal of $\lf$ such that $\leib(\lf)\leib(\lf)=0$.

Suppose now that $\leib(\lf)=\lf$. Then we have that $\lf\lf=0$. In particular, every square of
$\lf$ is zero. Consequently, we obtain that $\lf=\leib(\lf)=0$.
\end{proof}
\vspace{-.05cm}

\noindent {\bf Remark:} It is an immediate consequence of Proposition \ref{ker} that every
one-dimensional left or right Leibniz algebra is a Lie algebra. Moreover, with a little more work
one can use Proposition \ref{ker} to classify the two-dimensional left or right Leibniz algebras
(see \cite[pp.\ 11/12]{DMS1}).
\vspace{.2cm}

Let $\lf$ be a left or right Leibniz algebra. Then by definition of the Leibniz kernel, $\lf_\lie:=
\lf/\leib(\lf)$ is a Lie algebra. We call $\lf_\lie$ the {\em canonical Lie algebra\/} associated to
$\lf$. In fact, the Leibniz kernel is the smallest ideal such that the corresponding factor algebra
is a Lie algebra (see \cite[Theorem~2.5]{FM}).

\begin{pro}\label{minlie}
Let $\lf$ be a left or right Leibniz algebra. Then $\leib(\lf)$ is the smallest ideal $\If$ of $\lf$ such
that $\lf/\If$ is a Lie algebra.
\end{pro}

\begin{proof}
Let $\If$ be an ideal of $\lf$ such that $\lf/\If$ is a Lie algebra. Then it follows from
the anti-commutativity of $\lf/\If$ that $x^2\in\If$ for every element $x\in\lf$. Since $\If$ is a
subspace of $\lf$, we conclude that $\leib(\lf)\subseteq\If$.
\end{proof}

\begin{pro}\label{leftmult}
Let $\lf$ be a left Leibniz algebra. Then the set $L(\lf)$ of left multiplication operators of $\lf$ is
an ideal of the derivation algebra $\der(\lf)$ of $\lf$, and $\lf/C_\ell(\lf)$ is a Lie algebra that is
isomorphic to $L(\lf)$.
\end{pro}

\begin{proof}
By definition of a left Leibniz algebra, $L(\lf)\subseteq\der(\lf)$. Since the multiplication of $\lf$ is
bilinear, $L(\lf)$ is a subspace of $\der(\lf)$.

Let $x,y\in\lf$ and $D\in\der(\lf)$ be arbitrary elements. Then the computation $$[D,L_x](y)=(D
\circ L_x-L_x\circ D)(y)=D(xy)-xD(y)=D(x)y=L_{D(x)}(y)$$ shows that $[D,L_x]=L_{D(x)}\in\der
(\lf)$ for any element $x\in\lf$ and any derivation $D\in\der(\lf)$. This completes the proof that
$L(\lf)$ is an ideal of $\der(\lf)$.

Now consider the linear transformation $L:\lf\to\der(\lf)$ defined by $x\mapsto L_x$. By definition,
$\Ker(L)=C_\ell(\lf)$ and $\im(L)=L(\lf)$. Finally, it follows from the identity $[D,L_y]=L_{D(y)}$
that $L_{xy}=L_{L_x(y)}=[L_x,L_y]$ for any elements $x,y\in\lf$, and thus $L$ is a homomorphism
of Leibniz algebras. Then the fundamental homomorphism theorem shows that $L$ induces a Leibniz
algebra isomorphism from $L/C_\ell(\lf)$ onto $L(\lf)$.
\end{proof}

\noindent {\bf Remark.} I am very grateful to Friedrich Wagemann for pointing out to me the four-term
exact sequence of left Leibniz algebras $$0\to C_\ell(\lf)\to\lf\stackrel{L}\to\der(\lf)\to\out_\ell(\lf)\to
0\,,$$ where $\out_\ell(\lf):=\der(\lf)/L(\lf)$ is the {\em Lie algebra of outer derivations\/} of the left
Leibniz algebra $\lf$, and in which only $\lf$ is not necessarily a Lie algebra (see \cite[Proposition 1.8]{DW}).
\vspace{.3cm}

By combining Proposition \ref{leftker} and Proposition \ref{leftmult} with the fundamental homomorphism
theorem we obtain the following result.

\begin{cor}\label{leftleiblie}
Let $\lf$ be a left Leibniz algebra. Then the set $L(\lf)$ of left multiplication operators of $\lf$ is a
homomorphic image of the canonical Lie algebra $\lf_\lie$ associated to $\lf$.
\end{cor}

The first part of the next result is \cite[Theorem 2.3]{FM}.

\begin{pro}\label{rightmult}
Let $\lf$ be a right Leibniz algebra. Then the set $R(\lf)$ of right multiplication operators of $\lf$
is an ideal of the derivation algebra $\der(\lf)$ of $\lf$, and $\lf/C_r(\lf)$ is a Lie algebra that is
isomorphic to $R(\lf)^\op$.
\end{pro}

\begin{proof}
The proof of Proposition \ref{rightmult} is very similar to the proof of Proposition~\ref{leftmult}.
Nevertheless, we include the whole argument in this case as well, since it shows why left multiplication
operators are preferable when one writes functions on the left of their arguments.

By definition of a right Leibniz algebra, $R(\lf)\subseteq\der(\lf)$. Since the multiplication of $\lf$ is
bilinear, $R(\lf)$ is a subspace of $\der(\lf)$.

Let $x,y\in\lf$ and $D\in\der(\lf)$ be arbitrary elements. Then the computation $$[D,R_x](y)=(D
\circ R_x-R_x\circ D)(y)=D(yx)-D(y)x=yD(x)=R_{D(x)}(y)$$ shows that $[D,R_x]=R_{D(x)}\in
\der(\lf)$ for any element $x\in\lf$ and any derivation $D\in\der(\lf)$. This completes the proof
that $R(\lf)$ is an ideal of $\der(\lf)$.

Now consider the linear transformation $R:\lf\to\der(\lf)$ defined by $x\mapsto R_x$. By definition,
$\Ker(R)=C_r(\lf)$ and $\im(R)=R(\lf)$. Finally, it follows from the identity $[D,R_x]=R_{D(x)}$
that $R_{xy}=R_{R_y(x)}=[R_y,R_x]$ for any elements $x,y\in\lf$, and thus $R:\lf\to\der(\lf)^\op$
is a homomorphism of Leibniz algebras. Then the fundamental homomorphism theorem shows that
$R$ induces a Leibniz algebra isomorphism from $L/C_r(\lf)$ onto $R(\lf)^\op$.
\end{proof}

\noindent {\bf Remark.} Similarly to \cite[Proposition 1.8]{DW}, we obtain from the proof of
Proposition~\ref{rightmult} the four-term exact sequence of right Leibniz algebras $$0\to C_r(\lf)
\to\lf\stackrel{R}\to\der(\lf)^\op\to\out_r(\lf)^\op\to 0\,,$$ where $\out_r(\lf):=\der(\lf)/R(\lf)$
is the {\em Lie algebra of outer derivations\/} of the right Leibniz algebra $\lf$, and in which only
$\lf$ is not necessarily a Lie algebra.
\vspace{.3cm}

By combining Proposition \ref{rightker} and Proposition \ref{rightmult} with the fundamental
homomorphism theorem we obtain the following result.

\begin{cor}\label{rightleiblie}
Let $\lf$ be a right Leibniz algebra. Then the set $R(\lf)$ of right multiplication operators of $\lf$ is
a homomorphic image of the opposite of the canonical Lie algebra $\lf_\lie^\op$ associated to $\lf$.
\end{cor}

The Lie multiplication algebra of a symmetric Leibniz algebra can be described as follows. This generalizes
the corresponding result for Lie algebras (see \cite[p.~21]{S}).

\begin{thm}\label{liemult}
The Lie multiplication algebra $\lie(\lf)=L(\lf)+R(\lf)$ of a symmetric Leibniz algebra $\lf$ is an ideal
of the derivation algebra $\der(\lf)$ of $\lf$.
\end{thm}

\begin{proof}
Recall that the Lie multiplication algebra $\lie(\lf)$ of the symmetric Leibniz algebra $\lf$ is the
smallest subalgebra of the general linear Lie algebra $\gl(\lf)$ that contains $L(\lf)\cup R(\lf)$.
It follows from Proposition \ref{leftmult} and Proposition \ref{rightmult} that $L(\lf)$ and $R(\lf)$
are ideals of $\der(\lf)$. But as a sum of ideals, $L(\lf)+R(\lf)$ is an ideal of $\der(\lf)$, and
thus a subalgebra of $\der(\lf)$ which in turn is a subalgebra of $\gl(\lf)$. This shows that
$\lie(\lf)=L(\lf)+R(\lf)$.
\end{proof}

\noindent {\bf Example.} Let $\nf$ be the two-dimensional nilpotent symmetric Leibniz algebra
from Example~2. Then $L(\nf)=\F f=R(\nf)$, and therefore $\lie(\nf)=\F f$. On the other hand,
$\leib(\nf)=\F e$, and thus $\nf_\lie\cong\F f\cong\lie(\nf)$.
\vspace{.3cm}

\noindent {\bf Question.} What is the relationship between the canonical Lie algebra $\lf_\lie$
associated to a symmetric Leibniz algebra $\lf$ and the Lie multiplication algebra $\lie(\lf)$ of
$\lf$? Note that if $\lf$ is a Lie algebra, then $\lie(\lf)\cong\lf/C(\lf)$ (see \cite[Theorem~1.1.2\,(4)]{SF}
or Proposition \ref{leftmult}). So in this case $\lf_\lie\ne\lie(\lf)$ if $C(\lf)\ne 0$. But as
$\leib(\nf)=C(\nf)$, the isomorphism $\lie(\lf)\cong\lf/C(\lf)$ is compatible with the previous
example. 
\vspace{.3cm}

The following results generalize \cite[Exercise I.1.6]{Bou} from Lie algebras to left and right Leibniz
algebras. As in the case of Lie algebras they are an immediate consequence of the left and right
Leibniz identity, respectively. Note the mixture of left and right in both propositions which we will
see again in Proposition \ref{leftextnilp} and Proposition \ref{rightextnilp}.

\begin{pro}\label{leftass}
A left Leibniz algebra $\lf$ is associative if, and only if, $\lf^2\subseteq C_r(\lf)$.
\end{pro}

\begin{pro}\label{rightass}
A right Leibniz algebra $\lf$ is associative if, and only if, $\lf^2\subseteq C_\ell(\lf)$.
\end{pro}

These results show that, in general, left and right Leibniz algebras are far from being associative.
In fact, it follows from Proposition \ref{leftextnilp} and Proposition \ref{rightextnilp} below that
associative left/right Leibniz algebras are nilpotent.\footnote{Note that in dimension 3 there are
five isomorphism classes of nilpotent non-Lie Leibniz algebras of which four (including one 1-parameter
family) are associative (see \cite[Theorem 6.4]{DMS1}). In dimension 4 there are seventeen isomorphism
classes of indecomposable nilpotent non-Lie Leibniz algebras of which eleven (including three 1-parameter
families) are associative (see \cite[Theorem~3.2]{AOR}, but compare this with \cite{DMS3} for the
correct total number of isomorphism classes).}
\vspace{.2cm}

\noindent {\bf Examples.}
\begin{enumerate}
\item[(1)] For the two-dimensional solvable left Leibniz algebra $\af_\ell$ from Example 1 we have
                that $\af_\ell^2=\F f$ and $C_r(\af_\ell)=\F e$. Hence $\af_\ell$ is not associative.
\item[(2)] For the two-dimensional nilpotent symmetric Leibniz algebra $\nf$ from Example 2 we
                have that $\nf^2=\F e=C_\ell(\nf)=C_r(\nf)$. Hence $\nf$ is associative.
\item[(3)] For the five-dimensional simple left Leibniz algebra $\Sf_\ell$ from Example 3 we have
                that $\Sf_\ell^2=\Sf_\ell$ and $C_r(\Sf_\ell)=0$. Hence $\Sf_\ell$ is not associative.
\end{enumerate}
\vspace{.2cm}

In the sentence after the proof of \cite[Corollary 1.3]{B1} Barnes claims that left Leibniz
algebras are power-associative. But the following example shows that this is not always
the case. Let $\af_\ell$ denote the two-dimensional left Leibniz algebra from Example 1.
Then $$(e+f)^2(e+f)=0\ne f=(e+f)(e+f)^2\,,$$ which yields that $\af_\ell$ is not
power-associative.\footnote{This example also shows that left or right Leibniz algebras
are not necessarily flexible.}

The following result shows that Barnes' claim holds for symmetric Leibniz algebras. In fact,
like Lie algebras, such algebras are flexible.

\begin{pro}\label{flexible}
Every symmetric Leibniz algebra is flexible and power-associative. Moreover, $x^3=0$ for
every element $x$.
\end{pro}

\begin{proof}
Let $\lf$ be a symmetric Leibniz algebra. It follows from Lemma \ref{rightmultsq} that
$yx^2=0$ for any elements $x$ and $y$ of a right Leibniz algebra. Hence the left Leibniz
identity for $\lf$ yields $x(yx)=(xy)x+yx^2=(xy)x$ for any elements $x$ and $y$ of $\lf$,
and thus $\lf$ is flexible.

According to Lemma \ref{leftmultsq}, we have that $x^2y=0$ for any elements $x$ and
$y$ of a left Leibniz algebra. This implies that $x^2x=0=xx^2$ for any element $x\in\lf$.
In particular, third powers are well-defined and are equal to zero. Consequently, we obtain
by induction that symmetric Leibniz algebras are power-associative.\footnote{The proof
of Proposition \ref{flexible} shows that this result holds more generally for left central
Leibniz algebras or right central Leibniz algebras. The latter notions were introduced by
Mason and Yamskulna in \cite{MY}: A {\em left central Leibniz algebra\/} is a left Leibniz
algebra $\lf$ such that $\leib(\lf)\subseteq C_r(\lf)$ and a  {\em right central Leibniz
algebra\/} is a right Leibniz algebra $\lf$ such that $\leib(\lf)\subseteq C_\ell(\lf)$. It
follows from Proposition \ref{leftker} and Proposition \ref{rightker} that a left (resp.\ right)
Leibniz algebra is left (resp.\ right) central if, and only if, $\leib(\lf)\subseteq C(\lf)$.}
\end{proof}

So in this respect symmetric Leibniz algebras are only slightly more general than Lie algebras;
instead of squares of arbitrary elements of a Lie algebra being zero, cubes of arbitrary elements
of a symmetric Leibniz algebra are zero.

In general, the terms of the derived series of an algebra are only subalgebras. We finish
this section by proving that any term in the derived series of a left or a right Leibniz algebra
is indeed an ideal.

\begin{pro}\label{derser}
Let $\lf$ be a left or right Leibniz algebra. Then $\lf^{(n)}$ is an ideal of $\lf$ for every
non-negative integer $n$.
\end{pro}

\begin{proof}
We only prove the result for left Leibniz algebras as this yields the result for right Leibniz
algebras by considering the opposite algebra. We first prove that $\lf^{(n)}$ is a left ideal
of $\lf$ for every non-negative integer $n$. We proceed by induction on $n$. The base
step $n=0$ is clear as $\lf\lf^{(0)}=\lf\lf\subseteq\lf=\lf^{(0)}$. For the induction step let
$n>0$ be an integer and assume that the statement is true for $n-1$. Then we obtain from
the left Leibniz identity and the induction hypothesis that
\begin{eqnarray*}
\lf\lf^{(n)} & = & \lf[\lf^{(n-1)}\lf^{(n-1)}]\subseteq[\lf\lf^{(n-1)}]\lf^{(n-1)}+\lf^{(n-1)}
[\lf\lf^{(n-1)}]\\
& \subseteq & \lf^{(n-1)}\lf^{(n-1)}=\lf^{(n)}\,.
\end{eqnarray*}

Next, we prove that $\lf^{(n)}$ is a right ideal of $\lf$ for every non-negative integer $n$.
We again proceed by induction on $n$. The base step $n=0$ is clear as $\lf^{(0)}\lf=\lf\lf
\subseteq\lf=\lf^{(0)}$. For the induction step let $n>0$ be an integer and assume that
the statement is true for $n-1$. Then we obtain from identity (\ref{RLLI}) and the induction
hypothesis that
\begin{eqnarray*}
\lf^{(n)}\lf & = & [\lf^{(n-1)}\lf^{(n-1)}]\lf\subseteq\lf^{(n-1)}[\lf^{(n-1)}\lf]\\
& \subseteq & \lf^{(n-1)}\lf^{(n-1)}=\lf^{(n)}\,.
\end{eqnarray*}
This completes the proof.
\end{proof}


\section{Leibniz modules}


In this and in the next section we consider only left Leibniz algebras. We leave it to the interested
reader to formulate the corresponding definitions and results for right Leibniz algebras (see
\cite{L,LP1}).

Let $\lf$ be a left Leibniz algebra over a field $\F$. A {\em left $\lf$-module\/} is a vector space
$M$ over $\F$ with an $\F$-bilinear left $\lf$-action $\lf\times M\to M$, $(x,m)\mapsto x\cdot m$
such that $$(xy)\cdot m=x\cdot(y\cdot m)-y\cdot(x\cdot m)$$ is satisfied for every $m\in M$
and all $x,y\in\lf$.

The usual definitions of the notions of {\em submodule\/}, {\em irreducibility\/}, {\em complete
reducibility\/}, {\em composition series\/}, {\em homomorphism\/}, {\em isomorphism\/}, etc.,
hold for left Leibniz modules.

Moreover, every left $\lf$-module $M$ gives rise to a homomorphism $\lambda:\lf\to\gl(M)$ of
left Leibniz algebras, defined by $\lambda_x(m):=x\cdot m$, and vice versa. We call $\lambda$
the {\em left representation\/} of $\lf$ associated to $M$. We call the kernel of $\lambda$ the
{\em annihilator\/} of $M$ and denote it by $\ann_\lf(M)$.
\vspace{.2cm}

\noindent {\bf Examples.}
\begin{enumerate}
\item[(1)] Every left Leibniz algebra is a left module over itself via the Leibniz multiplication.
                This module is called the {\em left adjoint module\/} and will be denoted by $\lf_{
                \ad,\ell}$. The associated representation $L:\lf\to\gl(\lf)$ is called the {\em left
                adjoint representation\/} of $\lf$. Note that $\ann_\lf(\lf_{\ad,\ell})=C_\ell(\lf)$.
\item[(2)] The ground field $\F$ of any left Leibniz algebra is a left $\lf$-module via $x\cdot
                \alpha:=0$ for every element $x\in\lf$ and every scalar $\alpha\in\F$. This module
                is called the {\em trivial left module\/} of $\lf$.
\end{enumerate}
\vspace{.2cm}

The next result generalizes the second part of Proposition \ref{leftker} and reduces the study
of left Leibniz modules to Lie modules.

\begin{lem}\label{leftann}
Let $\lf$ be a left Leibniz algebra, and let $M$ be a left $\lf$-module. Then $\leib(\lf)\subseteq
\ann_\lf(M)$. 
\end{lem}

\begin{proof}
Since $\leib(\lf)$ is a subspace of $\lf$, it is enough to show that $x^2\in\ann_\lf(M)$ for any
element $x\in\lf$. We have that $\lambda_{x^2}=\lambda_x\circ\lambda_x-\lambda_x\circ
\lambda_x=0$, and therefore it follows that $x^2\in\Ker(\lambda)=\ann_\lf(M)$.
\end{proof}

By virtue of Lemma \ref{leftann}, every left $\lf$-module is an $\lf_\lie$-module, and vice
versa. This is the reason that in \cite[D\'efinition 1.1.14]{B} left Leibniz modules are called
Lie modules. Consequently, many properties of left $\lf$-modules follow from the corresponding
properties of $\lf_\lie$-modules. As one application of this point of view, we discuss trace
forms associated to finite-dimensional left Leibniz modules (see also \cite{AAO2,DMS1}).

For any finite-dimensional left representation $\lambda:\lf\to\gl(M)$ of a left Leibniz algebra
$\lf$ we define its {\em trace form\/} $\kappa_\lambda:\lf\times\lf\to\F$ by $\kappa_\lambda
(x,y):=\tr(\lambda_x\circ\lambda_y)$. Every trace form is an invariant symmetric bilinear form.
As usual, a bilinear form $\beta:\lf\times\lf\to\F$ is called {\it invariant\/} if $\beta(xy,z)=\beta
(x,yz)$ for any elements $x,y,z\in\lf$. The subspace $$\lf^\lambda:=\{x\in\lf\mid\forall\,y\in\lf:
\kappa_\lambda(x,y)=0\}$$ of $\lf$ is called the {\em radical\/} of $\kappa_\lambda$.
\vspace{.2cm}

\noindent {\bf Example.} As for Lie algebras, the trace form $\kappa:=\kappa_L$ associated
to the left adjoint representation $L$ of a finite-dimensional left Leibniz algebra $\lf$ is called
the {\em Killing form\/} of $\lf$. We will denote its radical by $\lf^\perp$.
\vspace{.2cm}

The next result is a consequence of Lemma \ref{leftann} and the invariance as well as the
symmetry of $\kappa_\lambda$.

\begin{lem}\label{leftrad}
Let $\lf$ be a left Leibniz algebra, and let $\lambda:\lf\to\gl(M)$ be a finite-dimensional left
representation of $\lf$. Then $\lf^\lambda$ is an ideal of $\lf$ with $\leib(\lf)\subseteq
\lf^\lambda$.
\end{lem}

\begin{proof}
The inclusion $\leib(\lf)\subseteq\lf^\lambda$ follows immediately from Lemma \ref{leftann}.

Let $x,y\in\lf$ and $z\in\lf^\lambda$ be arbitrary elements. Then we obtain from the invariance of
$\kappa_\lambda$ that $\kappa_\lambda(zx,y)=\kappa_\lambda(z,xy)=0$, i.e., $zx\in\lf^\lambda$.
This shows that $\lf^\lambda$ is a right ideal of $\lf$.

Moreover, using the symmetry in conjunction with the invariance of $\kappa_\lambda$, we conclude
that $\kappa_\lambda(xz,y)=\kappa_\lambda(y,xz)=\kappa_\lambda(yx,z)=\kappa_\lambda(z,yx)
=0$, i.e., $xz\in\lf^\lambda$. This proves that $\lf^\lambda$ is a left ideal of $\lf$.
\end{proof}

Lemma \ref{leftrad} implies that a trace form on a left non-Lie Leibniz algebra is never non-degenerate.
We call a trace form associated to a finite-dimensional left representation $\lambda$ of a left Leibniz
algebra $\lf$ {\em minimally degenerate\/}\footnote{In \cite[Definition 5.6]{DMS1} a minimally degenerate
Killing form is called {\em non-degenerate\/}, but this contradicts the usual definition of a non-degenerate
bilinear form. Note that non-Lie Leibniz algebras can admit non-degenerate invariant symmetric bilinear
forms, which, of course, cannot be trace forms. For example, this is the case for the two-dimensional
nilpotent symmetric Leibniz algebra $\nf$ from Example 2 in Section 2 (see \cite[Example 2.2]{BH}).} if
$\lf^\lambda=\leib(\lf)$.
\vspace{.2cm}

\noindent {\bf Example.} Let $\af_\ell$ be the two-dimensional solvable left Leibniz algebra from
Example 1 in Section 2. Recall that $\leib(\af_\ell)=\F f\ne 0$. Since $\kappa(e,e)=1\ne 0$ yields
$\af_\ell^\perp\ne\af_\ell$, we conclude from Lemma \ref{leftrad} that $\af_\ell^\perp=\leib
(\af_\ell)$. Hence the Killing form of $\af_\ell$ is minimally degenerate.
\vspace{.2cm}

The correct concept of a module for left Leibniz algebras is the notion of a Leibniz bimodule.
In order to motivate the appropriate definition of a bimodule for a left Leibniz algebra, we
follow the approach that Eilenberg proposed for any given class of non-associative algebras
(see \cite{E} and \cite[pp.\ 25/26]{S}).

Let $\lf$ be a left Leibniz algebra over a field $\F$, and let $M$ be a vector space over the
same ground field. Then the Cartesian product $\lf\times M$ with componentwise addition
and componentwise scalar multiplication is a vector space over $\F$. Suppose that $\lf$ acts
on $M$ from the left via the $\F$-bilinear map $\lf\times M\to M$, $(x,m)\mapsto x\cdot m$
and $\lf$ acts on $M$ from the right via the $\F$-bilinear map $M\times\lf\to M$, $(m,x)\mapsto
m\cdot x$. Then we define a multiplication on $\lf\times M$ by $$(x_1,m_1)(x_2,m_2):=(x_1x_2,
m_1\cdot x_2+x_1\cdot m_2)$$ for any elements $x_1,x_2\in\lf$ and $m_1,m_2\in M$. To
ensure that $\lf\times M$ satisfies the left Leibniz identity, we compute
\begin{eqnarray*}
(x_1,m_1)[(x_2,m_2)(x_3,m_3)] & = & (x_1,m_1)(x_2x_3,m_2\cdot x_3+x_2\cdot m_3)\\
& = & (x_1(x_2x_3),m_1\cdot(x_2x_3)+x_1\cdot(m_2\cdot x_3+x_2\cdot m_3))\\
& = & (x_1(x_2x_3),m_1\cdot(x_2x_3)+x_1\cdot(m_2\cdot x_3)+x_1\cdot(x_2\cdot m_3))\,,
\end{eqnarray*}
\begin{eqnarray*}
[(x_1,m_1)(x_2,m_2)](x_3,m_3) & = & (x_1x_2,m_1\cdot x_2+x_1\cdot m_2)(x_3,m_3)\\
& = & ((x_1x_2)x_3,(m_1\cdot x_2+x_1\cdot m_2)\cdot x_3+(x_1x_2)\cdot m_3))\\
& = & ((x_1x_2)x_3,(m_1\cdot x_2)\cdot x_3+(x_1\cdot m_2)\cdot x_3+(x_1x_2)\cdot m_3))\,,
\end{eqnarray*}
\begin{eqnarray*}
(x_2,m_2)[(x_1,m_1)(x_3,m_3)] & = & (x_2,m_2)(x_1x_3,m_1\cdot x_3+x_1\cdot m_3)\\
& = & (x_2(x_1x_3),m_2\cdot(x_1x_3)+x_2\cdot(m_1\cdot x_3+x_1\cdot m_3))\\
& = & (x_2(x_1x_3),m_2\cdot(x_1x_3)+x_2\cdot(m_1\cdot x_3)+x_2\cdot(x_1\cdot m_3))
\end{eqnarray*}
for any elements $x_1,x_2,x_3\in\lf$ and $m_1,m_2,m_3\in M$. The vector space $\lf\times M$
satisfies the left Leibniz identity if, and only if, the left-hand side of the first identity equals the
sum of the left-hand sides of the second and third identities. Hence the right-hand side of the
first identity must equal the sum of the right-hand sides of the second and third identities. For
the first components of the right-hand sides this is just the left Leibniz identity for $\lf$. From the
desired equality for the second components of the right-hand sides one can read off the following
three identities which are sufficient for $\lf\times M$ to satisfy the left Leibniz identity.
\begin{equation}\label{LLM}
x\cdot(y\cdot m)=(xy)\cdot m+y\cdot(x\cdot m)
\end{equation}
\begin{equation}\label{LML}
x\cdot(m\cdot y)=(x\cdot m)\cdot y+m\cdot(xy)
\end{equation}
\begin{equation}\label{MLL}
m\cdot(xy)=(m\cdot x)\cdot y+x\cdot(m\cdot y)
\end{equation}
for every $m\in M$ and all $x,y\in\lf$.

This motivates the following definition. An {\em $\lf$-bimodule\/} is a vector space $M$ with
a bilinear left $\lf$-action and a bilinear right $\lf$-action such that the following compatibility
conditions are satisfied:
\begin{enumerate}
\item[(LLM)] \hspace{2.5cm}$(xy)\cdot m=x\cdot(y\cdot m)-y\cdot(x\cdot m)$
\item[(LML)] \hspace{2.5cm}$(x\cdot m)\cdot y=x\cdot(m\cdot y)-m\cdot(xy)$
\item[(MLL)] \hspace{2.5cm}$(m\cdot x)\cdot y=m\cdot(xy)-x\cdot(m\cdot y)$
\end{enumerate}
for every $m\in M$ and all $x,y\in\lf$.

It is an immediate consequence of (LLM) that every Leibniz bimodule is a left Leibniz module.
Moreover, by combining (LML) and (MLL) we obtain the following result.

\begin{lem}\label{LRRR}
Let $\lf$ be a left Leibniz algebra, and let $M$ be an $\lf$-bimodule. Then $(x\cdot m)\cdot
y+(m\cdot x)\cdot y=0$ holds for every $m\in M$ and all $x,y\in\lf$. 
\end{lem}

As for left Leibniz modules, the usual definitions of the notions of {\em subbimodule\/}, {\em
irreducibility\/}, {\em complete reducibility\/}, {\em composition series\/}, {\em homomorphism\/},
{\em isomorphism\/}, etc., hold for Leibniz bimodules.

Let $\lf$ be a left Leibniz algebra over a field $\F$, and let $V$ be a vector space over $\F$.
Then a pair $(\lambda,\rho)$ of linear transformations $\lambda:\lf\to\End_\F(V)$ and $\rho:
\lf\to\End_\F(V)$ is called a {\em representation\/} of $\lf$ on $V$ if the following conditions
are satisfied:
\begin{equation}\label{LLMrep}
\lambda_{xy}=\lambda_x\circ\lambda_y-\lambda_y\circ\lambda_x
\end{equation}
\begin{equation}\label{LMLrep}
\rho_{xy}=\lambda_x\circ\rho_y-\rho_y\circ\lambda_x
\end{equation}
\begin{equation}\label{MLLrep}
\rho_{xy}=\rho_y\circ\rho_x+\lambda_x\circ\rho_y
\end{equation}
for every $m\in M$ and all $x,y\in\lf$.

Then every $\lf$-bimodule $M$ gives rise to a representation $(\lambda,\rho)$ of $\lf$ on $M$
via $\lambda_x(m):=x\cdot m$ and $\rho_x(m):=m\cdot x$. Conversely, every representation
$(\lambda,\rho)$ of $\lf$ on $M$ defines an $\lf$-bimodule structure on $M$ via $x\cdot m:=
\lambda_x(m)$ and $m\cdot x:=\rho_x(m)$.

In order to avoid confusion, the {\em annihilator\/} of an $\lf$-bimodule $M$ will be denoted by
$\ann_\lf^\bi(M)$, and it is defined as $\ann_\lf^\bi(M):=\Ker(\lambda)\cap\Ker(\rho)$.

An $\lf$-bimodule $M$ is called {\em symmetric\/} if $m\cdot x=-x\cdot m$ for every $x\in\lf$
and every $m\in M$. An $\lf$-bimodule $M$ is called {\em anti-symmetric\/} if $m\cdot x=0$
for every $x\in\lf$ and every $m\in M$. Moreover, an $\lf$-bimodule $M$ is called {\em trivial\/}
if $x\cdot m=0=m\cdot x$ for every $x\in\lf$ and every $m\in M$. Note that an $\lf$-bimodule
$M$ is trivial if, and only if, $M$ is symmetric and anti-symmetric. Moreover, $\ann_\lf^\bi(M)=
\ann_\lf(M)$ if $M$ is symmetric or anti-symmetric and $\ann_\lf^\bi(M)=\lf$ if, and only if, $M$
is trivial.
\vspace{.2cm}

\noindent {\bf Examples.}
\begin{enumerate}
\item[(1)] Every left Leibniz algebra $\lf$ is a bimodule over itself via the Leibniz multiplication,
                the so-called {\em adjoint bimodule\/} $\lf_\ad$ of $\lf$. The associated representation
                is $(L,R)$, where $L$ denotes the left multiplication operator of $\lf$ and $R$ denotes
                the right multiplication operator of $\lf$. This representation is called the {\em adjoint
                representation\/} of $\lf$. Note that $\ann_\lf^\bi(\lf_\ad)=C(\lf)$. Moreover, unless
                the ground field has characteristic two, the adjoint bimodule $\lf_\ad$ is symmetric if,
                and only if, $\lf$ is a Lie algebra.\footnote{In particular, the adjoint bimodule of a
                symmetric Leibniz algebra is not always symmetric as the name might suggest.}
\item[(2)] The ground field $\F$ of any left Leibniz algebra $\lf$ is a trivial $\lf$-bimodule via
                $x\cdot\alpha:=0=:\alpha\cdot x$ for every element $x\in\lf$ and every scalar
                $\alpha\in\F$.
\end{enumerate}
\vspace{.1cm}

\begin{lem}\label{ann}
Let $\lf$ be a left Leibniz algebra, and let $M$ be an anti-symmetric or a symmetric $\lf$-bimodule.
Then $\leib(\lf)\subseteq\ann_\lf^\bi(M)$. 
\end{lem}

\begin{proof}
Since every $\lf$-bimodule is a left $\lf$-module, it follows from Lemma \ref{leftann} that $\leib(\lf)
\subseteq\Ker(\lambda)$. If $M$ is anti-symmetric, then $\rho=0$, and the assertion is clear.
On the other hand, if $M$ is symmetric, then $\lambda_x=-\rho_x$ for every element $x\in\lf$.
Hence $\Ker(\rho)=\Ker(\lambda)$, and therefore $\leib(\lf)\subseteq\Ker(\rho)$ follows from
the already established inclusion $\leib(\lf)\subseteq\Ker(\lambda)$.
\end{proof}

Let $\lf$ be a left Leibniz algebra over a field $\F$, and let $M$ be an $\lf$-bimodule. We call $$M_0
:=\langle x\cdot m+m\cdot x\mid x\in\lf,m\in M\rangle_\F$$ the {\em anti-symmetric kernel\/} of $M$
(see \cite[p.~145]{LP1}).
\vspace{.2cm}

\noindent {\bf Example} (cf.\ also \cite[p.\ 479]{FM}){\bf .}
Let $\lf$ be a left Leibniz algebra. Then the identity $$xy+yx=(x+y)^2-x^2-y^2$$ shows that
$(\lf_\ad)_0\subseteq\leib(\lf)$. Moreover, if the ground field of $\lf$ has characteristic $\ne 2$,
then the identity $x^2=\frac{1}{2}(xx+xx)$ shows that $(\lf_\ad)_0=\leib(\lf)$.
\vspace{.2cm}

The next result generalizes Proposition \ref{leftker} to arbitrary Leibniz bimodules (see \cite[p.~145]{LP1}
and also the proof of Theorem 3.1 in \cite{FM}).

\begin{pro}\label{antisymker}
Let $\lf$ be a left Leibniz algebra, and let $M$ be an $\lf$-bimodule. Then $M_0$ is an anti-symmetric
$\lf$-subbimodule of $M$.
\end{pro}

\begin{proof}
Let $x,y\in\lf$ and $m\in M$ be arbitrary elements. Then we obtain from (\ref{LLM}) and (\ref{LML})
that
\begin{eqnarray*}
x\cdot(y\cdot m+m\cdot y) & = & x\cdot(y\cdot m)+x\cdot(m\cdot y)\\
& = & [(xy)\cdot m+m\cdot(xy)]+[y\cdot(x\cdot m)+(x\cdot m)\cdot y]\in M_0\,.
\end{eqnarray*}
Since $M_0$ is a subspace of $M$, we conclude that $\lf M_0\subseteq M_0$. Moreover, it is an
immediate consequence of Lemma \ref{LRRR} that $M_0\lf=0$.
\end{proof}

By definition of the anti-symmetric kernel, $M_\sym:=M/M_0$ is a symmetric $\lf$-bimodule. We
call $M_\sym$ the {\em symmetrization\/} of $M$. In fact, the anti-symmetric kernel is the smallest
subbimodule such that the corresponding factor module is symmetric. (This should be compared
with the analogous statement for the adjoint Leibniz bimodule in Proposition \ref{minlie}.)

\begin{pro}\label{sym}
Let $\lf$ be a left Leibniz algebra, and let $M$ be an $\lf$-bimodule. Then $M_0$ is the smallest
$\lf$-subbimodule $N$ of $M$ such that $M/N$ is a symmetric $\lf$-bimodule.
\end{pro}

\begin{proof}
Let $N$ be a $\lf$-subbimodule of $M$ such that $M/N$ is a symmetric $\lf$-bimodule.
Then it follows from the symmetry of $M/N$ that $x\cdot(m+N)=-(m+N)\cdot x$, and consequently,
$x\cdot m+m\cdot x\in N$ for every element $x\in\lf$ and every element $m\in M$. Since $N$ is a
subspace of $M$, we conclude that $M_0\subseteq N$.
\end{proof}

The next result was first established for finite-dimensional irreducible Leibniz bimodules over
finite-dimensional right or left Leibniz algebras (see \cite[Theorem~3.1]{LP2} and \cite[Theorem
1.4]{B2}). Then in \cite[Theorem 3.1]{FM} this was generalized to arbitrary irreducible Leibniz
bimodules over any right Leibniz algebra. At the same time the proof is simplified considerably.

\begin{thm}\label{irrbimod}
Let $\lf$ be a left Leibniz algebra. Then every irreducible $\lf$-bimodule is symmetric or anti-symmetric.
\end{thm}

\begin{proof}
Let $M$ be an irreducible $\lf$-bimodule. According to Proposition \ref{antisymker}, we have that $M_0
=0$ or $M_0=M$. In the former case $M$ is symmetric and in the latter case $M$ is anti-symmetric.
\end{proof}

The final result of this section discusses how a left Leibniz module can be made into an anti-symmetric
Leibniz bimodule or into a symmetric Leibniz bimodule.

\begin{pro}\label{leftmod}
Let $\lf$ be a left Leibniz algebra. Then the following statements hold:
\begin{enumerate}
\item[{\rm (a)}] Every left $\lf$-module with a trivial right action is an anti-symmetric $\lf$-bimodule.
\item[{\rm (b)}] Every left $\lf$-module determines a unique symmetric $\lf$-bimodule.
\end{enumerate}
\end{pro}

\begin{proof}
(a): Since (LLM) holds for every left $\lf$-module, and each term in (LML) as well as in (MLL)
contains at least one trivial right action, the assertion follows.

(b): In order for $M$ to be a symmetric $\lf$-bimodule, the right action of $\lf$ on $M$ has to
be defined by $m\cdot x:=-x\cdot m$ for every element $m\in M$ and every element $x\in\lf$.
Since (LLM) is automatically satisfied, one only needs to verify (LML) and (MLL). Let $x,y\in\lf$
and $m\in M$ be arbitrary elements. Then it follows from (LLM) that
\begin{eqnarray*}
(x\cdot m)\cdot y & = & -y\cdot(x\cdot m)\\
& = & (xy)\cdot m-x\cdot(y\cdot m)\\
& = & x\cdot(m\cdot y)-m\cdot(xy)
\end{eqnarray*}
and
\begin{eqnarray*}
(m\cdot x)\cdot y & = & y\cdot(x\cdot m)\\
& = & x\cdot(y\cdot m)-(xy)\cdot m\\
& = & m\cdot(xy)-x\cdot(m\cdot y)\,.
\end{eqnarray*}
This completes the proof.
\end{proof}
\vspace{.2cm}

Note that for the trivial left Leibniz module the Leibniz bimodule structures obtained from parts (a)
and (b) of Proposition \ref{leftmod} both give rise to the trivial Leibniz bimodule.
\vspace{.2cm}

\noindent {\bf Examples.} Let $\lf$ be a left Leibniz algebra and consider the left adjoint $\lf$-module
$\lf_{\ad,\ell}$.
\begin{enumerate}
\item[(1)] According to Proposition \ref{leftmod}\,(a), $\lf_{\ad,\ell}$ with a trivial right action is an
                anti-symmetric $\lf$-bimodule, which we call the {\em anti-symmetric adjoint bimodule\/}
                $\lf_a$ of $\lf$. The associated representation is $(L,0)$, the so-called {\em anti-symmetric
                adjoint representation\/} of $\lf$, where $L$ denotes the left multiplication operator of $\lf$.
\item[(2)] By virtue of Proposition \ref{leftmod}\,(b), $\lf_{\ad,\ell}$ has a unique symmetric
                $\lf$-bimodule structure, which we call the {\em symmetric adjoint bimodule\/} $\lf_s$ of
                $\lf$. The associated representation is $(L,-L)$, the so-called {\em symmetric adjoint
                representation\/} of $\lf$.
\end{enumerate}
\vspace{-.1cm}


\section{Leibniz cohomology}


Let $\lf$ be a left Leibniz algebra over a field $\F$, and let $M$ be an $\lf$-bimodule. For any
non-negative integer $n$ set $\CL^n(\lf,M):=\Hom_\F(\lf^{\otimes n},M)$ and consider the linear
transformation $d\hspace{.4mm}^n:\CL^n(\lf,M)\to\CL^{n+1}(\lf,M)$ defined by
\begin{eqnarray*}
(d\hspace{.4mm}^nf)(x_1\otimes\cdots\otimes x_{n+1}) & := & \sum_{i=1}^n(-1)^{i+1}x_i
\cdot f(x_1\otimes\cdots\otimes\hat{x}_i\otimes\cdots\otimes x_{n+1})\\
& + & (-1)^{n+1}f(x_1\otimes\cdots\otimes x_n)\cdot x_{n+1}\\
& + & \sum_{1\le i<j\le n+1}(-1)^if(x_1\otimes\cdots\otimes\hat{x}_i\otimes\cdots\otimes x_i
x_j\otimes\cdots\otimes x_{n+1})
\end{eqnarray*}
for any $f\in\CL^n(\lf,M)$ and all elements $x_1,\dots,x_{n+1}\in\lf$.

It is proved in \cite[Lemma 1.3.1]{Co} that $(\CL^n(\lf,M),d\hspace{.4mm}^n)_{n\in\N_0}$
is a cochain complex, i.e., $d\hspace{.4mm}^{n+1}\circ d\hspace{.4mm}^n=0$ for every
non-negative integer $n$. Of course, the original idea of defining Leibniz cohomology as
the cohomology of such a cochain complex for right Leibniz algebras is due to Loday
\cite[(1.8)]{LP1} (see also \cite{Bl2}). Hence one can define the {\em cohomology of $\lf$
with coefficients in\/} $M$ by $$\HL^n(\lf,M):=\HCE^n(\CL^\bullet(\lf,M),d\hspace{.4mm}^\bullet)
:=\Ker(d\hspace{.4mm}^n)/\im(d\hspace{.4mm}^{n-1})$$ for every non-negative integer $n$.
(Note that $d^{-1}:=0$.)

We will describe now the cohomology spaces in degree $0$ and $1$ (see \cite[Section~1.3.3]{Co}
for left Leibniz algebras and \cite[(1.8)]{LP1} for right Leibniz algebras). For this purpose we need
to introduce more notation.

The subspace $M^\lf:=\{m\in M\mid\forall\,x\in\lf:m\cdot x=0\}$ of an $\lf$-bimodule $M$ is
called the {\em space of right $\lf$-invariants\/} of $M$.

\begin{pro}\label{inv}
Let $\lf$ be a left Leibniz algebra, and let $M$ be an $\lf$-bimodule. Then $\HL^0(\lf,M)=M^\lf$.
\end{pro}

\begin{proof}
We have that $(d\hspace{.4mm}^0m)(x)=-m\cdot x$ for any $m\in M$ and any $x\in\lf$. Hence
$\HL^0(\lf,M)=\Ker(d\hspace{.4mm}^0)=M^\lf$.
\end{proof}

In particular, we obtain from Lemma \ref{ann} and the definition of $M^\lf$ the following
result.

\begin{cor}\label{asinv}
Let $\lf$ be a left Leibniz algebra, and let $M$ be an $\lf$-bimodule. Then the following statements
hold:
\begin{enumerate}
\item[{\rm (a)}] If $M$ is symmetric, then $\HL^0(\lf,M)\cong M^{\lf_\lie}$.
\item[{\rm (b)}] If $M$ is anti-symmetric, then $\HL^0(\lf,M)=M$.
\end{enumerate}
\end{cor}

The subspace $$\der(\lf,M):=\{f\in\Hom_\F(\lf,M)\mid\forall\,x,y\in\lf:f(xy)=f(x)\cdot y+x
\cdot f(y)\}$$ of $\Hom_\F(\lf,M)$ is called the {\em space of derivations from $\lf$ to $M$\/}.
The subspace $\ider(\lf,M):=\{g\in\Hom_\F(\lf,M)\mid\exists\,m\in M:g(x)=m\cdot x\}$ of
$\der(\lf,M)$ is called {\em the space of inner derivations from $\lf$ to $M$\/}.

\begin{pro}\label{outder}
Let $\lf$ be a left Leibniz algebra, and let $M$ be an $\lf$-bimodule. Then $\HL^1(\lf,M)=\der
(\lf,M)/\ider(\lf,M)$.
\end{pro}

\begin{proof}
It follows from $(d\hspace{.4mm}^0m)(x)=-m\cdot x$ that $\ider(\lf,M)=\im(d\hspace{.4mm}^0)$.
Moreover, we have that $(d\hspace{.2mm}^1 f)(x\otimes y)=x\cdot f(y)+f(x)\cdot y-f(xy)$ for
any $m\in M$, $f\in\Hom_\F(\lf,M)$, and $x,y\in\lf$. Hence $\Ker(d\hspace{.2mm}^1)=\der(\lf,M)$,
and therefore $\HL^1(\lf,M)=\Ker(d\hspace{.2mm}^1)/\im(d\hspace{.4mm}^0)=\der(\lf,M)/\ider
(\lf,M)$.
\end{proof}

\begin{cor}\label{asoutder}
Let $\lf$ be a left Leibniz algebra over a field $\F$, and let $M$ be an $\lf$-bimodule. Then the following
statements hold:
\begin{enumerate}
\item[{\rm (a)}] If $M$ is symmetric, then $\HL^1(\lf,M)\cong\HL^1(\lf_\lie,M)=\HCE^1(\lf_\lie,M)$.
\item[{\rm (b)}] If $M$ is anti-symmetric, then $\HL^1(\lf,M)=\Hom_\lf(\lf,M)$, where the latter
                          denotes the vector space of homomorphisms of left $\lf$-modules.
\end{enumerate}
\end{cor}

\begin{proof}
(a): Let $f\in\der(\lf,M)$ and let $x\in\lf$ be arbitrary. Since $M$ is symmetric, we have that $f(x^2)
=f(x)\cdot x+x\cdot f(x)=-x\cdot f(x)+x\cdot f(x)=0$. Since $f$ is linear, it follows that $f(\leib(\lf))
=0$. Hence $f$ induces a derivation from $\lf_\lie$ to $M$. This yields the isomorphism $\HL^1(\lf,M)
\cong\HL^1(\lf_\lie,M)$.

Let $d\hspace{.4mm}^n_\lie$ denote the coboundary map for the Chevalley-Eilenberg cohomology
of Lie algebras (see \cite[Exercise I.3.12]{Bou}). Then we have that $$(d\hspace{.4mm}^1_\lie f)
(x\wedge y)=x\cdot f(y)-y\cdot f(x)-f(xy)$$ for any $m\in M$, any $f\in\Hom_\F(\lf_\lie,M)$, and
any $x,y\in\lf_\lie$. Since $M$ is symmetric, we obtain that $\Ker(d\hspace{.4mm}^1_\lie)=\der(
\lf_\lie,M)$. Similarly, $(d\hspace{.4mm}^0_\lie m)(x)=x\cdot m$ for any $m\in M$ and any $x\in
\lf_\lie$. Hence $\im(d\hspace{.4mm}^0_\lie)=\ider(\lf_\lie,M)$, and thus we have that $\HL^1
(\lf_\lie,M)=\HCE^1(\lf_\lie,M)$.

(b): Since $M$ is anti-symmetric, we have that $$\der(\lf,M)=\{f\in\Hom_\F(\lf,M)\mid\forall\,x,y
\in\lf:f(xy)=x\cdot f(y)\}$$ as well as $\ider(\lf,M)=\{0\}$, and therefore it follows from Proposition
\ref{outder} that $\HL^1(\lf,M)=\Hom_\lf(\lf,M)$.
\end{proof}

In particular, we obtain from Corollary \ref{asoutder} the following result.

\begin{cor}\label{triv}
If $\lf$ is a left Leibniz algebra over a field $\F$, then $$\HL^1(\lf,\F)\cong\lf/\lf^2\cong\lf_\lie/
\lf_\lie^2\,.$$
\end{cor}

The next result is reminiscent of Seligman's non-vanishing theorem for modular Lie algebras (see
\cite[p. 102]{Se}).

\begin{cor}\label{asadj}
If $\lf\ne 0$ is a left Leibniz algebra, then $\HL^1(\lf,\lf_a)\ne 0$.
\end{cor}

\begin{proof}
It follows from Corollary \ref{asoutder}\,(b) that $0\ne\id_\lf\in\Hom_\lf(\lf,\lf_a)\cong\HL^1
(\lf,\lf_a)$.
\end{proof}

Finally, we briefly discuss abelian extensions of left Leibniz algebras.
Let $\lf$ and $\kf$ be left Leibniz algebras. Then any short exact sequence $0\to\kf\to\ef\to
\lf\to 0$ is called an {\em extension of $\lf$ by $\kf$\/}, and $\kf$ is called the {\em kernel\/}
of the extension. Two extensions $0\to\kf\stackrel{\iota}\to\ef\stackrel{\pi}\to\lf\to 0$ and
$0\to\kf\stackrel{\iota^\prime}\to\ef^\prime\stackrel{\pi^\prime}\to\lf\to 0$ of $\lf$ by $\kf$
are called {\em equivalent\/} if there is a homomorphism $\phi:\ef\to\ef^\prime$ such that
$\phi\circ\iota=\iota^\prime$ and $\pi^\prime\circ\phi=\pi$. Every extension $0\to\kf\to\ef
\to\lf\to 0$ with abelian kernel $\kf$ gives rise to an $\lf$-bimodule structure on $\kf$ (see
\cite[Proposition 1.3.6]{Co}). It can be shown that the set of equivalence classes of all
extensions of $\lf$ by a given $\lf$-bimodule $M$ (considered as an abelian Leibniz algebra)
are in natural bijection to $\HL^2(\lf,M)$ (see \cite[Theorem 1.3.13]{Co} for left Leibniz
algebras and \cite[Proposition 1.9]{LP1} for right Leibniz algebras). We will only prove the
following ingredient of establishing such a bijection.

\begin{lem}\label{ext}
Let $\lf$ be a left Leibniz algebra, let $M$ be an $\lf$-bimodule, and let $f\in\CL^2(\lf,M)$ be
a Leibniz $2$-cocycle. Then $\ef:=\lf\times M$ with Leibniz product defined by $$(x_1,m_1)
(x_2,m_2):=(x_1x_2,m_1\cdot x_2+x_1\cdot m_2+f(x_1\otimes x_2))$$ for any $x_1,x_2\in
\lf$ and any $m_1,m_2\in M$ is a left Leibniz algebra.
\end{lem}

\begin{proof}
Since the Cartesian product $\lf\times M$ with componentwise addition and componentwise scalar
multiplication is a vector space, we only need to verify the left Leibniz identity. We have that
\begin{eqnarray*}
(d\hspace{.2mm}^2 f)(x_1\otimes x_2\otimes x_3) & = & x_1\cdot f(x_2\otimes x_3)-x_2\cdot
f(x_1\otimes x_3)-f(x_1\otimes x_2)\cdot x_3\\
& & -f(x_1x_2\otimes x_3)+f(x_1\otimes x_2x_3)-f(x_2\otimes x_1x_3)
\end{eqnarray*}
for any $x_1,x_2,x_3\in\lf$.
Since $f$ is a $2$-cocycle, we have that $$x_1\cdot f(x_2\otimes x_3)+f(x_1\otimes x_2x_3)=
x_2\cdot f(x_1\otimes x_3)+f(x_1\otimes x_2)\cdot x_3+f(x_1x_2\otimes x_3)+f(x_2\otimes x_1x_3)\,.$$

We compute
\begin{eqnarray*}
(x_1,m_1)[(x_2,m_2)(x_3,m_3)] & = & (x_1,m_1)(x_2x_3,m_2\cdot x_3+x_2\cdot m_3+f(x_2\otimes x_3))\\
& = & (x_1(x_2x_3),m_1\cdot(x_2x_3)+x_1\cdot(m_2\cdot x_3)+x_1\cdot(x_2\cdot m_3)\\
& & +\,x_1\cdot f(x_2\otimes x_3)+f(x_1\otimes x_2x_3))\,,
\end{eqnarray*}
\begin{eqnarray*}
[(x_1,m_1)(x_2,m_2)](x_3,m_3) & = & (x_1x_2,m_1\cdot x_2+x_1\cdot m_2+f(x_1\otimes x_2))(x_3,m_3)\\
& = & ((x_1x_2)x_3,(m_1\cdot x_2)\cdot x_3+(x_1\cdot m_2)\cdot x_3\\
& & +\,f(x_1\otimes x_2)\cdot x_3+(x_1x_2)\cdot m_3+f(x_1x_2\otimes x_3))\,,
\end{eqnarray*}
\begin{eqnarray*}
(x_2,m_2)[(x_1,m_1)(x_3,m_3)] & = & (x_2,m_2)(x_1x_3,m_1\cdot x_3+x_1\cdot m_3+f(x_1\otimes x_3))\\
& = & (x_2(x_1x_3),m_2\cdot(x_1x_3)+x_2\cdot(m_1\cdot x_3)+x_2\cdot(x_1\cdot m_3)\\
& & +\,x_2\cdot f(x_1\otimes x_3)+f(x_2\otimes x_1x_3))
\end{eqnarray*}
for any elements $x_1,x_2,x_3\in\lf$ and $m_1,m_2,m_3\in M$.

It follows from the left Leibniz identity, (\ref{LLM}), (\ref{LML}), (\ref{MLL}), and the $2$-cocycle
identity for $f$ that the right--hand side of the first identity equals the sum of the right-hand sides
of the second and third identities. Hence the left-hand side of the first identity also equals the
sum of the left-hand sides of the second and third identities, and thus $\lf\times M$ satisfies the
left Leibniz identity.
\end{proof}


\section{Nilpotent Leibniz algebras}


Let $\lf$ be a left or right Leibniz algebra. Then the {\it left descending central series\/} $$^1\hspace{-.5mm}
\lf\supseteq\hspace{.1mm}^2\hspace{-.5mm}\lf\supseteq\hspace{.1mm}^3\hspace{-.5mm}\lf
\supseteq\cdots$$ of $\lf$ is defined recursively by $^1\hspace{-.5mm}\lf:=\lf$ and $^{n+1}
\hspace{-.5mm}\lf:=\lf\hspace{.5mm}^n\hspace{-.5mm}\lf$ for every positive integer $n$.
Similarly, the {\it right descending central series\/} $$\lf^1\supseteq\lf^2\supseteq\lf^3\supseteq
\cdots$$ of $\lf$ is defined recursively by $\lf^1:=\lf$ and $\lf^{n+1}:=\lf^n\lf$ for every
positive integer $n$.\footnote{This definition is consistent with the definition of $\lf^2$ given
in Section 1.} Note that Proposition \ref{rightlowcenser} is just \cite[Lemma 1]{AO1}. For the
convenience of the reader we include its proof.

\begin{pro}\label{leftlowcenser}
Let $\lf$ be a left Leibniz algebra. Then $(^n\hspace{-.5mm}\lf)_{n\in\N}$ is a descending filtration
of $\lf$, i.e., $^n\hspace{-.5mm}\lf\supseteq\hspace{.1mm}^{n+1}\hspace{-.5mm}\lf$ and
$^m\hspace{-.5mm}\lf\hspace{.5mm}^n\hspace{-.5mm}\lf\subseteq\hspace{.1mm}^{m+n}
\hspace{-.5mm}\lf$ for all positive integers $m$ and $n$. In particular, $^n\hspace{-.5mm}\lf$
is an ideal of $\lf$ for every positive integer $n$.
\end{pro}

\begin{pro}\label{rightlowcenser}
Let $\lf$ be a right Leibniz algebra. Then $(\lf^n)_{n\in\N}$ is a descending filtration of $\lf$,
i.e., $\lf^n\supseteq\lf^{n+1}$ and $\lf^m\lf^n\subseteq\lf^{m+n}$ for all positive integers
$m$ and $n$. In particular, $\lf^n$ is an ideal of $\lf$ for every positive integer $n$.
\end{pro}

\begin{proof}
We only prove Proposition \ref{rightlowcenser} as this yields Proposition \ref{leftlowcenser}
by considering the opposite algebra. Firstly, we show that the right descending central series
is indeed descending, i.e., $\lf^{n+1}\subseteq\lf^n$ for every positive integer $n$. We
proceed by induction on $n$. The base step $n=1$ is clear. For the induction step let $n>1$
be an integer and assume that the statement is true for $n-1$. Then $$\lf^{n+1}=\lf^n\lf
\subseteq\lf^{n-1}\lf=\lf^n\,.$$

Next, we prove $\lf^m\lf^n\subseteq\lf^{m+n}$ by induction on $n$. The base step $n=1$
is an immediate consequence of the definition: $$\lf^n\lf^1=\lf^n\lf=\lf^{n+1}\,.$$ For the
induction step let $n>1$ be an integer and assume that the statement is true for $n-1$. It
follows from identity (\ref{LRLI}) and by applying the induction hypothesis twice that
\begin{eqnarray*}
\lf^m\lf^n & = & \lf^m(\lf^{n-1}\lf)\subseteq(\lf^m\lf^{n-1})\lf+(\lf^m\lf)\lf^{n-1}\\
& \subseteq & \lf^{m+n-1}\lf+\lf^{m+1}\lf^{n-1}\subseteq\lf^{m+n}\,.
\end{eqnarray*}
This completes the proof.
\end{proof}

Following \cite[p.\ 2]{AO1} and \cite[p. 3829]{P1} we define the {\it descending central series\/}
$$\lf_1\supseteq\lf_2\supseteq\lf_3\supseteq\cdots$$ of a left or right Leibniz algebra $\lf$
recursively by $\lf_1:=\lf$ and $\lf_n:=\sum\limits_{k=1}^{n-1}\lf_k\lf_{n-k}$ for every
positive integer $n$. Note that Proposition \ref{lowcenserright} is just \cite[Lemma 2]{AO1}.
For the convenience of the reader we include its proof.

\begin{pro}\label{lowcenserleft}
Let $\lf$ be a left Leibniz algebra. Then $\lf_n=\hspace{.1mm}^n\hspace{-.5mm}\lf$ for
every positive integer $n$. In particular, $(\lf_n)_{n\in\N}$ is a descending filtration of $\lf$,
i.e., $\lf_n\supseteq\lf_{n+1}$ and $\lf_m\lf_n\subseteq\lf_{m+n}$ for all positive integers
$m$ and $n$.
\end{pro}

\begin{pro}\label{lowcenserright}
Let $\lf$ be a right Leibniz algebra. Then $\lf_n=\lf^n$ for every positive integer $n$. In
particular, $(\lf_n)_{n\in\N}$ is a descending filtration of $\lf$, i.e., $\lf_n\supseteq\lf_{n+1}$
and $\lf_m\lf_n\subseteq\lf_{m+n}$ for all positive integers $m$ and $n$.
\end{pro}

\begin{proof}
We only prove the first statement in Proposition \ref{lowcenserright} as this in conjunction
with Proposition \ref{rightlowcenser} implies the remaining statements. We proceed by induction
on $n$. The base step $n=1$ is an immediate consequence of the definitions. For the induction
step let $n>1$ be an integer and assume that the statement is true for any integer less than
$n$. Then by applying the induction hypothesis and Proposition \ref{rightlowcenser} we obtain
that $$\lf_n=\sum_{k=1}^{n-1}\lf_k\lf_{n-k}=\sum_{k=1}^{n-1}\lf^k\lf^{n-k}\subseteq\lf^n
\,.$$ On the other hand, it follows from the induction hypothesis that $$\lf^n=\lf^{n-1}\lf=
\lf_{n-1}\lf_1\subseteq\sum_{k=1}^{n-1}\lf_k\lf_{n-k}=\lf_n\,.$$ This completes the proof.
\end{proof}

By combining Proposition \ref{lowcenserleft} and Proposition \ref{lowcenserright} we obtain the
following result (cf.\ also \cite[Proposition 2.13]{BH}).

\begin{cor}\label{lowcenser}
If $\lf$ is a symmetric Leibniz algebra, then $\lf_n=\hspace{.1mm}^n\hspace{-.5mm}\lf=\lf^n$
for every positive integer $n$.
\end{cor}

Recall from Section 1 that an algebra $\lf$ is called {\em nilpotent\/} if there exists a positive
integer $n$ such that any product of $n$ elements in $\lf$, no matter how associated, is zero
(see \cite[p. 18]{S}). The following observation is clear.

\begin{lem}\label{defnilp}
A left or right Leibniz algebra $\lf$ is nilpotent if, and only if, there exists a positive integer $n$
such that $\lf_n=0$.
\end{lem}

\noindent {\bf Examples.}
\begin{enumerate}
\item[(1)] The two-dimensional solvable left Leibniz algebra $\af_\ell$ from Example 1 in Section
                2 is not nilpotent as $(\af_\ell)_n=\hspace{.1mm}^n\hspace{-.1mm}\af_\ell=\F f\ne
                0$ for every integer $n\ge 3$. But note that $\af_r^n=0$ for every integer $n\ge 3$.
\item[(2)] The two-dimensional symmetric Leibniz algebra $\nf$ from Example~2 in Section 2 is
                nilpotent as $\nf_n=\hspace{.1mm}^n\hspace{-.1mm}\nf=\nf^n=0$ for every integer
                $n\ge 3$.
\item[(3)] The five-dimensional simple left Leibniz algebra $\Sf_\ell$ from Example 3 in Section 2
                 is not nilpotent as $(\Sf_\ell)_n=\hspace{.1mm}^n\hspace{-.1mm}\Sf_\ell=\Sf_\ell^n
                =\Sf_\ell\ne 0$ for every positive integer $n$.
\end{enumerate}
\vspace{.1cm}

In the proof of Proposition \ref{subalghomimnilp} we will need the following result.

\begin{lem}\label{homlowcenser}
If $\phi:\lf\to\kf$ is a homomorphism of left or right Leibniz algebras, then $\phi(\lf_n)=\phi(\lf)_n$
for every positive integer $n$.
\end{lem}

\begin{proof}
We proceed by induction on $n$. The base step $n=1$ follows from $$\phi(\lf_1)=\phi(\lf)=\phi
(\lf)_1\,.$$ For the induction step let $n>1$ be an integer and assume that the statement is true
for any integer less than $n$. Then we obtain from the induction hypothesis that
\begin{eqnarray*}
\phi(\lf_n) & = & \phi(\sum_{k=1}^{n-1}\lf_k\lf_{n-k})=\sum_{k=1}^{n-1}\phi(\lf_k)\phi(\lf_{n-k})\\
& = & \sum_{k=1}^{n-1}\phi(\lf)_k\phi(\lf)_{n-k}=\phi(\lf)_n\,.
\end{eqnarray*}
This completes the proof.
\end{proof}

The next result is an immediate consequence of Lemma \ref{homlowcenser}.

\begin{pro}\label{subalghomimnilp}
Subalgebras and homomorphic images of nilpotent left or right Leibniz algebras are nilpotent.
\end{pro}

\begin{pro}\label{leftextnilp}
If $\If$ is an ideal of a left Leibniz algebra $\lf$ such that $\If\subseteq C_r(\lf)$ and $\lf/\If$ is
nilpotent, then $\lf$ is nilpotent.
\end{pro}

\begin{pro}\label{rightextnilp}
If $\If$ is an ideal of a right Leibniz algebra $\lf$ such that $\If\subseteq C_\ell(\lf)$ and $\lf/\If$
is nilpotent, then $\lf$ is nilpotent.
\end{pro}

\begin{proof}
We only prove Proposition \ref{rightextnilp} as this yields Proposition \ref{leftextnilp} by considering
the opposite algebra.

Since $\lf/\If$ is nilpotent, there exists a positive integer $r$ such that $\lf^r=\lf_r\subseteq\If$.
But by hypothesis, we have that $\If\subseteq C_\ell(\lf)$, i.e., $\If\lf$=0. Hence $\lf_{r+1}=
\lf^{r+1}=\lf^r\lf\subseteq\If\lf=0$, and therefore $\lf$ is nilpotent.
\end{proof}

\begin{pro}\label{sumidnilp}
The sum of two nilpotent ideals of a left or right Leibniz algebra is nilpotent.
\end{pro}

\begin{proof}
Let $\lf$ be a right Leibniz algebra. We begin by proving that $\If^n$ is a right ideal of $\lf$ for
any right ideal $\If$ of $\lf$ and for any positive integer $n$. We will proceed by induction on
$n$. The base step $\If^1\lf=\If\lf\subseteq\If=\If^1$ follows from the fact that $\If$ is a
right ideal of $\lf$. For the induction step let $n>1$ be an integer and assume that the statement
is true for $n-1$. Then we obtain from the right Leibniz identity and the induction hypothesis that
$$\If^n\lf=(\If^{n-1}\If)\lf\subseteq(\If^{n-1}\lf)\If+\If^{n-1}(\If\lf)\subseteq\If^{n-1}\If=
\If^n\,.$$

Now let $\jf$ and $\kf$ be two nilpotent ideals of $\lf$. We will prove that $$(\jf+\kf)^n\subseteq
\sum_{r=1}^{n-1}(\jf^r\cap\kf^{n-r})$$ for every positive integer $n$. We will again proceed by
induction on $n$. The base step $(\jf+\kf)^1=\jf+\kf=\jf^1+\kf^1$ is an immediate consequence
of the definition of the right descending central series. For the induction step let $n\ge 1$ be an integer
and assume that the statement is true for $n$. Then we obtain from the induction hypothesis and
the fact that $\jf^r$ and $\kf^{n-r}$ are right ideals of $\lf$:
\begin{eqnarray*}
(\jf+\kf)^{n+1} & = & (\jf+\kf)^n(\jf+\kf)\subseteq\sum_{r=1}^{n-1}(\jf^r\cap\kf^{n-r})(\jf+\kf)\\
& \subseteq & \sum_{r=1}^{n-1}(\jf^r\cap\kf^{n-r})\jf+\sum_{r=1}^{n-1}(\jf^r\cap\kf^{n-r})\kf\\
& \subseteq & \sum_{r=1}^{n-1}(\jf^{r+1}\cap\kf^{n-r})+\sum_{r=1}^{n-1}(\jf^r\cap\kf^{n+1-r})\\
& = & \sum_{s=2}^n(\jf^s\cap\kf^{n+1-s})+\sum_{r=1}^{n-1}(\jf^r\cap\kf^{n+1-r})\\
& = & \jf^n\cap\kf^1+\sum_{r=2}^{n-1}(\jf^r\cap\kf^{n+1-r})+\jf^1\cap\kf^n\\
& = & \sum_{r=1}^n(\jf^r\cap\kf^{n+1-r})\,.
\end{eqnarray*}
Finally, as $\jf$ and $\kf$ are nilpotent, there exist positive integers $s$ and $t$ such that $\jf^s
=0$ and $\kf^t=0$. Hence we obtain that $(\jf+\kf)^{s+t-1}=0$ which shows that $\jf+\kf$
is a nilpotent ideal of $\lf$.
\end{proof}

Using Proposition \ref{sumidnilp} one can proceed similar to the solvable case to establish the
existence of a largest nilpotent ideal (see Section 1). We leave the details to the interested
reader (see also \cite[Proposition 1]{G}).

Now we prove the Lie analogue of a consequence of \cite[Theorem 2.4]{S} for symmetric Leibniz
algebras.

\begin{thm}\label{nilp}
For every symmetric Leibniz algebra $\lf$ the following statements are equivalent:
\begin{enumerate}
\item[(i)]   $\lf$ is nilpotent.
\item[(ii)]  $\lf_\lie$ is nilpotent.
\item[(iii)]  $\lf/C(\lf)$ is nilpotent.
\item[(iv)] $\lie(\lf)$ is nilpotent.
\end{enumerate}
\end{thm}

\begin{proof}
The implication (i)$\Rightarrow$(ii) is an immediate consequence of Proposition \ref{subalghomimnilp}.
It follows from Corollary \ref{kercen} that there is a natural epimorphism $\lf_\lie\to\lf/C(\lf)$ of Lie
algebras. Hence another application of Proposition \ref{subalghomimnilp} yields the implication
(ii)$\Rightarrow$(iii). Moreover, the implication (iii)$\Rightarrow$(i) follows from Proposition~\ref{leftextnilp}.

Suppose now that $\lf/C(\lf)$ is nilpotent. Then we obtain from Proposition \ref{leftmult} that there
is a natural epimorphism $\lf/C(\lf)\to\lf/C_\ell(\lf)\cong L(\lf)$, and therefore $L(\lf)$ is nilpotent.
Similarly, we obtain from Proposition \ref{rightmult} that there is a natural epimorphism $\lf/C(\lf)
\to\lf/C_r(\lf)\cong R(\lf)^\op$, and thus $R(\lf)$ is nilpotent. Hence it follows from  Proposition
\ref{sumidnilp} and Theorem \ref{liemult} that $\lie(\lf)$ is nilpotent. This establishes the implication
(iii)$\Rightarrow$(iv). Finally, the implication (iv)$\Rightarrow$(i) can be obtained from Proposition
\ref{subalghomimnilp} in conjunction with Proposition \ref{leftmult} and Proposition~\ref{rightextnilp}.
\end{proof}

\noindent {\bf Example.} Let $\af_\ell$ denote the two-dimensional non-nilpotent solvable left
Leibniz algebra $\af_\ell$ from Example 1  in Section 2. As has been observed above, we have
that $C_\ell(\af_\ell)=\F f=\leib(\af_\ell)$ and $C_r(\af_\ell)=\F e$. Hence $\lf_\lie=\lf/\leib(\lf)$,
$L(\lf)\cong\lf/C_\ell(\lf)$, and $R(\lf)\cong[\lf/C_r(\lf)]^\op$ are one-dimensional. So each of
these Lie algebras is abelian, and thus nilpotent, but $\af_\ell$ is not nilpotent. This shows that
the implication (ii)$\Rightarrow$(i) in Theorem \ref{nilp} neither holds for left nor for right
Leibniz algebras. Moreover, neither of the implications (iii)$\Rightarrow$(i) and (iv)$\Rightarrow$(i)
in Theorem~\ref{nilp} holds for left (resp.\ right) Leibniz algebras if one replaces $C(\lf)$ by
$C_\ell(\lf)$ (resp.\ $C_r(\lf)$) and $\lie(\lf)$ by $L(\lf)$ (resp.\ $R(\lf)$).
\vspace{.2cm}

The Leibniz analogue of Engel's theorem for Lie algebras of linear transformations was first
proved by Patsourakos \cite[Theorem 7]{P1} and later by Barnes \cite[Theorem 1.2]{B2}
who used the corresponding result for Lie algebras in his proof. Note that Patsourakos does
not have to assume that the representation is finite-dimensional. For the convenience of the
reader we include a variation of Barnes' proof.

\begin{thm}\label{ej}
Let $\lf$ be a finite-dimensional  left Leibniz algebra, and let $(\lambda,\rho)$ be a representation
of $\lf$ on a non-zero finite-dimensional vector space $M$ such that $\lambda_x$ is nilpotent for
every element $x\in\lf$. Then $\rho_x$ is nilpotent for every element $x\in\lf$, and there exists
a non-zero vector $m\in M$ such that $\lambda_x(m)=0=\rho_x(m)$ for every element $x\in\lf$.
\end{thm}

\begin{proof} 
We first prove that for any element $x\in\lf$ the nilpotency of $\lambda_x$ implies the nilpotency
of $\rho_x$ (see also \cite[Lemma 6]{P1}). Namely, we have $\rho_x^n=(-1)^{n-1}\rho_x
\circ\lambda_x^{n-1}$ for every element $x\in\lf$ and every positive integer $n$. This can be
shown by induction on $n$. The base step $n=1$ is trivial. For the induction step let $n\ge 1$
and assume that the statement is true for $n$. Then it follows from the induction hypothesis and 
Lemma \ref{LRRR} that $$\rho_x^{n+1}=\rho_x\circ\rho_x^n=(-1)^{n-1}\rho_x\circ\rho_x\circ
\lambda_x^{n-1}=-(-1)^{n-1}\rho_x\circ\lambda_x\circ\lambda_x^{n-1}=(-1)^n\rho_x\circ
\lambda_x^n\,.$$

Since the $\lf$-bimodule $M$ is finite-dimensional, it has an irreducible $\lf$-subbi\-module $N$.
We obtain from Theorem \ref{irrbimod} that $N$ is symmetric or anti-symmetric. In the former
case we have that $\rho_x=-\lambda_x$, and in the latter case we have that $\rho_x=0$ for
every $x\in\lf$. It follows from the linearity of $\lambda$ and (\ref{LLMrep}) that $\lambda(\lf)$
is a Lie subalgebra of $\gl(M)$. So the existence of a non-zero vector $m\in N$ such that
$\lambda_x(m)=0$ for every $x\in\lf$ can be obtained from Engel's theorem for Lie algebras
of linear transformations (see \cite[Theorem 3.3]{H}). Finally, this and the symmetry or
anti-symmetry of $N$ yield $\rho_x(m)=0$ for every $x\in\lf$.
\end{proof}

We conclude this section by several applications of Theorem \ref{ej}. The first result is just a
reformulation of Theorem \ref{ej} in terms of a composition series of a Leibniz bimodule (see also
\cite[Corollary 9]{P1}).

\begin{cor}\label{striang}
Let $\lf$ be a finite-dimensional  left Leibniz algebra over a field $\F$, and let $(\lambda,\rho)$ be
a representation of $\lf$ on a non-zero finite-dimensional vector space $M$ such that $\lambda_x$
is nilpotent for every element $x\in\lf$. Then the following statements hold:
\begin{enumerate}
\item[(a)] If $M$ is irreducible, then $M$ is the one-dimensional trivial $\lf$-bimodule.
\item[(b)] If $M$ is finite-dimensional, then every composition series $$0=M_0\subsetneqq M_1
                \subsetneqq\cdots\subsetneqq M_n=M$$ of $M$ satisfies $\dim_\F M_j=j$, $\lambda_x
                (M_j)\subseteq M_{j-1}$, and $\rho_x(M_j)\subseteq M_{j-1}$ for every integer $j\in
                \{1,\dots,n\}$ and every element $x\in\lf$.
\end{enumerate}
\end{cor}

\begin{proof}
(a) is an immediate consequence of the proof of Theorem \ref{ej} as we did not use the finite
dimension of $M$ in the irreducible case.

(b): We apply part (a) to each composition factor $M_j/M_{j-1}$ ($1\le j\le n$).
\end{proof}

Next, we specialize Corollary \ref{striang} and Theorem \ref{ej} to the adjoint Leibniz bimodule.

\begin{cor}\label{adjbimodnilp}
Let $\lf$ be a finite-dimensional  nilpotent left Leibniz algebra over a field $\F$, and let $\If$ be a
$d$-dimensional ideal of $\lf$. Then the following statements hold:
\begin{enumerate}
\item[(a)] There exists an ascending chain $$0=\lf_0\subsetneqq\lf_1\subsetneqq\cdots\subsetneqq
                \lf_n=\lf$$ of ideals of $\lf$ such that $\If=\lf_d$, $\dim_\F\lf_j=j$, $\lf\lf_j\subseteq
                \lf_{j-1}$, and $\lf_j\lf\subseteq\lf_{j-1}$ for every integer $j\in\{1,\dots,n\}$.
\item[(b)] If $\If\ne 0$, then $\lf\If\subsetneqq\If$ and $\If\lf\subsetneqq\If$
\item[(c)] If $\If\ne 0$, then $\If\cap C(\lf)\ne 0$.
\end{enumerate}
\end{cor}

\begin{proof}
Since $\lf$ is nilpotent, we have that $L_x$ is nilpotent for every element $x\in\lf$.

(a): Choose a composition series of the adjoint $\lf$-bimodule that contains $\If$ (see \cite[Proposition
1.1.1]{SF}) and apply Corollary \ref{striang} to the adjoint representation $(L,R)$ of $\lf$.

(b) is an immediate consequence of part (a).

(c): Note that $\If$ is an $\lf$-subbimodule of the adjoint $\lf$-bimodule. It follows from Theorem
\ref{ej} that there exists a non-zero element $y\in\If$ such that $L_x(y)=0=R_x(y)$, i.e.,
$xy=0=yx$ for every element $x\in\lf$. Hence $0\ne y\in C_r(\lf)\cap C_\ell(\lf)=C(\lf)$.
\end{proof}

Finally, we prove Engel's theorem for left Leibniz algebras (see \cite[Theorem 2]{AO1} and \cite[Corollary
10]{P1}).

\begin{cor}\label{engel}
If $\lf$ is a finite-dimensional  left Leibniz algebra such that $L_x$ is nilpotent for every element
$x\in\lf$, then $\lf$ is nilpotent.
\end{cor}

\begin{proof}
In the notation of Corollary \ref{adjbimodnilp}, we have that $\hspace{.1mm}^k\hspace{-.5mm}\lf
\subseteq\lf_{n-k+1}$ for every positive integer $k$, and therefore $\hspace{.1mm}^{n+1}
\hspace{-.5mm}\lf=0$. We proceed by induction on $k$ to prove the former statement. The base
step $k=1$ is just $\hspace{.1mm}^1\hspace{-.5mm}\lf=\lf=\lf_n$. For the induction step let
$k>1$ be an integer and assume that $\hspace{.1mm}^{k-1}\hspace{-.5mm}\lf\subseteq
\lf_{n-k+2}$ is true. Then we obtain from the induction hypothesis in conjunction with Corollary
\ref{adjbimodnilp}\,(a) that $\hspace{.1mm}^k\hspace{-.5mm}\lf=\lf\hspace{.5mm}^{k-1}
\hspace{-.5mm}\lf\subseteq\lf\lf_{n-k+2}\subseteq\lf_{n-k+1}$.
\end{proof}


\section{Solvable Leibniz algebras}


In the case of solvable Leibniz algebras we have one-sided analogues of Theorem~\ref{nilp}.
These show that the solvability of a left or right Leibniz algebra $\lf$ is equivalent to the
solvability of several Lie algebras associated to $\lf$.

\begin{pro}\label{leftsolv}
For every left Leibniz algebra $\lf$ the following statements are equivalent:
\begin{enumerate}
\item[(i)]   $\lf$ is solvable.
\item[(ii)]  $\lf_\lie$ is solvable.
\item[(iii)] $\lf/C_\ell(\lf)$ is solvable.
\item[(iv)] $L(\lf)$ is solvable.
\end{enumerate}
\end{pro}

\begin{pro}\label{rightsolv}
For every right Leibniz algebra $\lf$ the following statements are equivalent:
\begin{enumerate}
\item[(i)]   $\lf$ is solvable.
\item[(ii)]  $\lf_\lie$ is solvable.
\item[(iii)] $\lf/C_r(\lf)$ is solvable.
\item[(iv)] $R(\lf)$ is solvable.
\end{enumerate}
\end{pro}

\begin{proof}
We only prove Proposition \ref{rightsolv} as this yields Proposition \ref{leftsolv} by considering the
opposite algebra.

The implication (i)$\Rightarrow$(ii) is an immediate consequence of Proposition \ref{subalghomimsolv}.
It follows from Proposition \ref{rightker} that there is a natural epimorphism $\lf_\lie=\lf/\leib(\lf)\to
\lf/C_r(\lf)$ of Lie algebras. Hence another application of Proposition \ref{subalghomimsolv} yields
the implication (ii)$\Rightarrow$(iii). Moreover, the implication (iii)$\Rightarrow$(i) follows from
Proposition~\ref{extsolv}. Finally, the remaining equivalence of (iii) and (iv) is an immediate
consequence of Proposition \ref{rightmult}.
\end{proof}

\begin{thm}\label{solv}
For every symmetric Leibniz algebra $\lf$ the following statements are equivalent:
\begin{enumerate}
\item[(i)]   $\lf$ is solvable.
\item[(ii)]  $\lf_\lie$ is solvable.
\item[(iii)]  $\lf/C(\lf)$ is solvable.
\item[(iv)] $L(\lf)$ is solvable.
\item[(v)] $R(\lf)$ is solvable.
\item[(vi)] $\lie(\lf)$ is solvable
\end{enumerate}
\end{thm}

\begin{proof}
The equivalence of (i) and (ii) follows from Proposition \ref{leftsolv} or Proposition~\ref{rightsolv}
and the implication (i)$\Rightarrow$(iii) is an immediate consequence of Proposition \ref{subalghomimsolv}.

Suppose now that $\lf/C(\lf)$ is solvable. Then we obtain from Proposition \ref{leftmult} that there
is a natural epimorphism $\lf/C(\lf)\to\lf/C_\ell(\lf)\cong L(\lf)$, and therefore $L(\lf)$ is solvable.
Similarly, we obtain from Proposition \ref{rightmult} that there is a natural epimorphism $\lf/C(\lf)
\to\lf/C_r(\lf)\cong R(\lf)^\op$, and thus $R(\lf)$ is solvable. This establishes the implications
(iii)$\Rightarrow$(iv) and (iii)$\Rightarrow$(v). Each of the implications (iv)$\Rightarrow$(i)
and (v)$\Rightarrow$(i) follows from Proposition \ref{leftsolv} and Proposition \ref{rightsolv},
respectively.

The implication (i)$\Rightarrow$(vi) is a consequence of Theorem \ref{liemult} and Proposition
\ref{sumidsolv} in conjunction with the already established implications. Finally, the implication
(vi)$\Rightarrow$(iv) can be obtained from Theorem \ref{liemult} and Proposition \ref{subalghomimsolv}.
\end{proof}

The Leibniz analogue of Lie's theorem for solvable Lie algebras of linear transformations was
proved by Patsourakos \cite[Theorem 2]{P2}. For the convenience of the reader we include
a proof that follows very closely the proof of Theorem \ref{ej} and uses the corresponding
result for Lie algebras.

\begin{thm}\label{lie}
Let $\lf$ be a finite-dimensional solvable left Leibniz algebra over an algebraically closed field
of characteristic zero, and let $(\lambda,\rho)$ be a representation of $\lf$ on a non-zero
finite-dimensional vector space $M$. Then $M$ contains a common eigenvector for the linear
transformations in $\lambda(\lf)\cup\rho(\lf)$.
\end{thm}

\begin{proof}
Since the $\lf$-bimodule $M$ is finite-dimensional, it has an irreducible $\lf$-sub\-module $N$.
We obtain from Theorem \ref{irrbimod} that $N$ is symmetric or anti-symmetric. In the former
case we have that $\rho_x=-\lambda_x$, and in the latter case we have that $\rho_x=0$ for
every $x\in\lf$. Similarly to the proof of Proposition \ref{leftmult}, one can show that $\lambda
(\lf)$ is a Lie subalgebra of $\gl(M)$ such that $\lambda(\lf)\cong\lf/\ann_\lf(M)$. The latter
isomorphism in conjunction with Proposition \ref{subalghomimsolv} yields that $\lambda(\lf)$ is
solvable. So the existence of a common eigenvector $m\in N$ for $\lambda(\lf)$ can be obtained
from Lie's theorem for solvable Lie algebras of linear transformations (see \cite[Theorem~4.1]{H}).
Finally, this and the symmetry or anti-symmetry of $N$ imply that $m$ is also a common eigenvector
for $\rho(\lf)$.
\end{proof}

We conclude this section by several applications of Theorem \ref{lie}. The first result is just a
reformulation of Theorem \ref{lie} in terms of a composition series of a Leibniz bimodule (see also
\cite[Corollary 2]{P2} and \cite[Theorem 3.2]{DMS1}).

\begin{cor}\label{triang}
Let $\lf$ be a finite-dimensional solvable left Leibniz algebra over an algebraically closed field
$\F$ of characteristic zero, and let $M$ be a non-zero finite-dimensional $\lf$-bimodule. Then the
following statements hold:
\begin{enumerate}
\item[(a)] If $M$ is irreducible, then $\dim_\F M=1$.
\item[(b)] Every composition series $$0=M_0\subsetneqq M_1\subsetneqq\cdots\subsetneqq
                M_n=M$$ of $M$ satisfies $\dim_\F M_j=j$ for any $1\le j\le n$.
\end{enumerate}
\end{cor}

\begin{proof}
(a) is an immediate consequence of the proof of Theorem \ref{lie}.

(b): We apply part (a) to each composition factor $M_j/M_{j-1}$ ($1\le j\le n$).
\end{proof}

\noindent {\bf Remark.} Corollary \ref{triang} (and thus also Theorem \ref{lie}) is not true for
ground fields of prime characteristic as already can be seen for the non-abelian two-dimensional
Lie algebra (see \cite[Example 5.9.1]{SF}). Moreover, the one-dimensional Lie algebra spanned
by any proper rotation of a two-dimensional real vector space shows that Corollary \ref{triang}
and Theorem \ref{lie} in general hold only over algebraically closed fields.
\vspace{-.1cm}

Next, we specialize Corollary \ref{triang} to the adjoint Leibniz bimodule (see also \cite[Corollary
3.3]{DMS1}).

\begin{cor}\label{adjbimodsolv}
Let $\lf$ be a finite-dimensional solvable left Leibniz algebra over an algebraically closed field $\F$
of characteristic zero, and let $\If$ be a $d$-dimensional ideal of $\lf$. Then there exists an
ascending chain $$0=\lf_0\subsetneqq\lf_1\subsetneqq\cdots\subsetneqq\lf_n=\lf$$ of ideals of
$\lf$ such that $\If=\lf_d$ and $\dim_\F\lf_j=j$ for every $1\le j\le n$.
\end{cor}

\begin{proof}
Choose a composition series of the adjoint $\lf$-bimodule that contains $\If$ (see \cite[Proposition
1.1.1]{SF}) and apply Corollary \ref{triang} to the adjoint $\lf$-bimodule.
\end{proof}

We can also employ Engel's theorem for left Leibniz algebras to prove the following result (see
\cite[Theorem 4]{AO1}, \cite[Corollary 3]{P2}, and \cite[Corollary 6]{G}).

\begin{cor}\label{derengel}
Let $\lf$ be a finite-dimensional solvable left Leibniz algebra over an algebraically closed field $\F$
of characteristic zero. Then $L_x$ is nilpotent for every element $x\in\lf^2$. In particular, $\lf^2$
is nilpotent.
\end{cor}

\begin{proof}
According to Corollary \ref{adjbimodsolv}, there exists an ascending chain $$0=\lf_0\subsetneqq
\lf_1\subsetneqq\cdots\subsetneqq\lf_n=\lf$$ of ideals of $\lf$ such that $\dim_\F\lf_j=j$ for every
$1\le j\le n$. Hence one can choose successively a basis $\{x_1,\dots,x_j\}$ of $\lf_j$ ($1\le j\le
n$) such that the corresponding matrices of $L(\lf)$ are upper triangular. By virtue of the proof of
Proposition \ref{leftmult}, we have that $L(\lf^2)=[L(\lf),L(\lf)]$, and thus the matrices of $L(\lf^2)$
are strictly upper triangular. Hence $L_x$ is nilpotent for every element $x\in\lf^2$. In particular,
we obtain from Corollary \ref{engel} that $\lf^2$ is nilpotent.
\end{proof}

Finally, we give a proof of Cartan's solvability criterion for left Leibniz algebras (see also \cite[Theorem
3.7]{AAO2} and \cite[Theorem 3.5]{DMS1}).

\begin{thm}\label{cartan}
Let $\lf$ be a finite-dimensional left Leibniz algebra over a field of characteristic zero. Then $\lf$
is solvable if, and only if, $\kappa(x,y)=0$ for every element $x\in\lf$ and every element $y\in
\lf^2$.
\end{thm}

\begin{proof}
Suppose first that $\lf$ is solvable and the ground field of $\lf$ is algebraically closed. It follows
from Corollary \ref{adjbimodsolv} that $L(\lf)$ can be simultaneously representated by upper
triangular matrices. Then the proof of Proposition \ref{leftmult} shows that the corresponding
matrix of $L_{yz}=[L_y,L_z]=L_y\circ L_z-L_z\circ L_y$ is a strictly upper triangular matrix for
any $y,z\in\lf$. Hence $\kappa(x,yz)=\tr(L_x\circ L_{yz})=0$ for any $x,y,z\in\lf$.

Suppose now that $\kappa(x,y)=0$ for every element $x\in\lf$ and every element $y\in\lf^2$,
and the ground field of $\lf$ is again algebraically closed. It follows from Proposition~\ref{leftmult}
that $L(\lf)$ is a Lie subalgebra of $\gl(\lf)$. In particular, we obtain as before that $\tr(L_x\circ
[L_y,L_z])=\tr(L_x\circ L_{yz})=\kappa(x,yz)=0$ for any $x,y,z\in\lf$. So the other implication
is a consequence of \cite[Theorem~4.3]{H} in conjunction with Proposition~\ref{leftsolv}.

Finally, in case the ground field $\F$ of $\lf$ is not algebraically closed, a base field extension
will show the assertion. Namely, let $\overline{\F}$ be an algebraic closure of $\F$, set
$\overline{\lf}:=\lf\otimes_\F\overline{\F}$, and let $\overline{L}_{a\otimes\alpha}(b\otimes
\beta):=ab\otimes\alpha\beta$ for any $a,b\in\lf$ and any $\alpha,\beta\in\overline{\F}$
denote the left multiplication operator of $\overline{\lf}$. Since $\overline{\lf}\hspace{.3mm}^2
=\lf^2\otimes_\F\overline{\F}$, we obtain by induction that $\overline{\lf}\hspace{.3mm}^{(n)}
=\lf^{(n)}\otimes_\F\overline{\F}$ for every non-negative integer $n$. Consequently, $\lf$
is solvable if, and only if, $\overline{\lf}$ is solvable. Moreover, we obtain for the Killing form
$\overline{\kappa}$ of $\overline{\lf}$ that $\overline{\kappa}(x\otimes 1,y\otimes 1)=\tr
(\overline{L}_{x\otimes 1}\circ\overline{L}_{y\otimes 1})=\tr(L_x\circ L_y)=\kappa(x,y)$
for all $x,y\in\lf$ as $\overline{L}_{a\otimes\alpha}(b\otimes\beta)=L_a(b)\otimes\alpha
\beta$ for any $a,b\in\lf$ and any $\alpha,\beta\in\overline{\F}$. This can be used to show
that $\kappa(x,y)=0$ for every element $x\in\lf$ and every element $y\in\lf^2$ if, and only
if, $\overline{\kappa}(\overline{x},\overline{y})=0$ for every element $\overline{x}\in\overline{\lf}$
and every element $\overline{y}\in\overline{\lf}\hspace{.3mm}^2$.
\end{proof}


\section{Semisimple Leibniz algebras}


A left or right Leibniz algebra $\lf$ is called {\em simple\/} if $0$, $\leib(\lf)$, $\lf$ are the
only ideals of $\lf$, and $\leib(\lf)\subsetneqq\lf^2$ (see \cite[Definition 1]{DA} for the
first part of this definition or \cite[Definition 2.2]{ORT}, \cite[Definition 5.1]{DMS1},
\cite[Definition 2.6]{FM}). Note that there is also another definition of simplicity for Leibniz
algebras in the literature, namely, requiring that $0$ and $\lf$ are the only ideals of $\lf$.
Then it follows from Proposition~\ref{ker} that every simple left or right Leibniz algebra has
to be a Lie algebra.

Exactly as for Lie algebras, we call a left or right Leibniz algebra $\lf$ {\em perfect\/} in case
$\lf=\lf^2$ holds. Then the first result in this section is an immediate consequence of these
definitions.

\begin{pro}\label{simperfect}
Every simple left or right Leibniz algebra is perfect. 
\end{pro}

The next result follows from the correspondence theorem for ideals. Note also that the condition
$\leib(\lf)\subsetneqq\lf^2$ implies that the canonical Lie algebra $\lf_\lie$ associated to a simple
left or right Leibniz algebra $\lf$ is not abelian. 

\begin{pro}\label{sim}
If $\lf$ is a simple left or right Leibniz algebra, then $\lf_\lie$ is a simple Lie algebra.
\end{pro}

\noindent We will see in the example after Theorem \ref{strucsemisim} that the converse of Proposition
\ref{sim} does not always hold.

In analogy with the above definition of simplicity, we call a left or right Leibniz algebra $\lf$ {\em
semisimple\/} when $\leib(\lf)$ contains every solvable ideal of $\lf$.

\begin{pro}\label{simsemisim}
Every simple left or right Leibniz algebra is semisimple.
\end{pro}

\begin{proof}
Let $\If$ be any solvable ideal of the simple left or right Leibniz algebra $\lf$. Then either
$\If=0$, $\If=\leib(\lf)$, or $\If=\lf$. In the first two cases we have that $\If\subseteq
\leib(\lf)$, and we are done. So suppose that $\If=\lf$. From Proposition~\ref{simperfect}
we obtain by induction that $\lf=\lf^{(n)}$ holds for every non-negative integer $n$. Since
by hypothesis $\If$ is solvable, there exists a non-negative integer $r$ such that $\If^{(r)}
=0$. Hence $\lf=\lf^{(r)}=\If^{(r)}=0$ which contradicts the requirement $\leib(\lf)
\subsetneqq\lf^2$.
\end{proof}

In the finite-dimensional case we have the following result which often is taken as the definition
of semisimplicity for Leibniz algebras (see \cite[Definition 5.2]{DMS1} or \cite[Definition 2.12]{FM}).

\begin{pro}\label{semisimrad}
A finite-dimensional left or right Leibniz algebra $\lf$ is semisimple if, and only if, $\leib(\lf)=
\rad(\lf)$.
\end{pro}

\begin{proof}
It follows from Proposition \ref{ker} and Proposition \ref{rad} that $\leib(\lf)\subseteq\rad(\lf)$.
This in conjunction with the semisimplicity of $\lf$ proves the ``only if"-part of the assertion, and
the converse follows from Proposition \ref{rad} which says that $\rad(\lf)$ is the largest solvable
ideal of $\lf$.
\end{proof}

\begin{pro}\label{semisimleftcen}
If $\lf$ is a semisimple left Leibniz algebra, then $\leib(\lf)=C_\ell(\lf)$. 
\end{pro}

\begin{pro}\label{semisimrightcen}
If $\lf$ is a semisimple right Leibniz algebra, then $\leib(\lf)=C_r(\lf)$. 
\end{pro}

\begin{cor}\label{semisimcen}
If $\lf$ is a semisimple symmetric Leibniz algebra, then $\leib(\lf)=C(\lf)$. 
\end{cor}

\begin{proof}
We only prove Proposition \ref{semisimrightcen} as this yields Proposition \ref{semisimleftcen}
by considering the opposite algebra.

It follows from Proposition \ref{rightker} that $\leib(\lf)\subseteq C_r(\lf)$. According to
Proposition~\ref{rightcen}, $C_r(\lf)$ is an abelian ideal of $\lf$, and thus the semisimplicity
of $\lf$ yields that $C_r(\lf)\subseteq\leib(\lf)$.
\end{proof}

\begin{pro}\label{semisim}
A left or right Leibniz algebra $\lf$ is semisimple if, and only if, $\mathfrak{L}_\mathrm{Lie}$
is semisimple.
\end{pro}

\begin{proof}
Suppose that $\lf$ is semisimple, and let $\IIf$ be any solvable ideal of the Lie algebra $\lf_\lie
=\lf/\leib(\lf)$. Let $\pi:\lf\to\lf_\lie$ denote the natural epimorphism of Leibniz algebras. Then
$\If:=\pi^{-1}(\IIf)$ is an ideal of $\lf$, and it follows from Proposition~\ref{extsolv} applied
to $\IIf=\If+\leib(\lf)/\leib(\lf)\cong\If/\If\cap\leib(\lf)$ that $\If$ is solvable. Hence the
semisimplicity of $\lf$ yields that $\If\subseteq\leib(\lf)$, i.e., $\IIf=0$. Consequently,
$\lf_\lie$ is semisimple.

In order to prove the converse, suppose that $\lf_\lie=\lf/\leib(\lf)$ is semisimple, and let
$\If$ be any solvable ideal of the left or right Leibniz algebra $\lf$. Then by Proposition
\ref{subalghomimsolv} we obtain that $\If+\leib(\lf)/\leib(\lf)$ is a solvable ideal of
$\lf_\lie$. Since the latter Lie algebra is semisimple, we have that $\If+\leib(\lf)/\leib(\lf)
=0$, and thus $\If\subseteq\leib(\lf)$. Hence $\lf$ is semisimple.
\end{proof}

Since finite-dimensional semisimple Lie algebras over a field of characteristic zero are perfect,
it follows from Proposition \ref{semisim} in conjunction with Corollary \ref{triv} that the same
is true for Leibniz algebras (see \cite[Corollary 5.5]{DMS1}).

\begin{cor}\label{semisimperfect}
Every finite-dimensional semisimple left Leibniz algebra over a field of characteristic zero is
perfect.
\end{cor}

Similar to the Killing form of a finite-dimensional semisimple Lie algebra over a field of characteristic
zero being non-degenerate, the Killing form of a finite-dimensional semisimple Leibniz algebra over
a field of characteristic zero is minimally degenerate (see \cite[Theorem 5.8]{DMS1}). As for Lie
algebras, this can be obtained from Cartan's solvability criterion for Leibniz algebras (see Theorem
\ref{cartan}).

\begin{pro}\label{mindeg}
The Killing form of every finite-dimensional semisimple left Leibniz algebra over a field of characteristic
zero is minimally degenerate.
\end{pro}

\begin{proof}
Let $\lf$ be a finite-dimensional semisimple left Leibniz algebra over a field of characteristic $0$,
and let $\If:=\lf^\perp$ denote the radical of the Killing form $\kappa$ of $\lf$. Recall that $\If$
is an ideal of $\lf$. The proof of \cite[Lemma 5.1]{H} shows that $\kappa_\If=\kappa_{\vert\If
\times\If}$, where $\kappa_\If$ denotes the Killing form of $\If$. Then we have that $\kappa_\If
(x,y)=\kappa(x,y)=0$ for every $x\in\If$ and every $y\in\If^2$. Hence Theorem \ref{cartan}
shows that $\If$ is solvable. Since $\lf$ is semisimple, we obtain that $\lf^\perp=\If\subseteq
\leib(\lf)$, and thus it follows from Lemma \ref{leftrad} that $\lf^\perp=\leib(\lf)$, i.e., $\kappa$
is minimally degenerate.
\end{proof}
\vspace{-.1cm}

\noindent {\bf Remark.} The example after Lemma \ref{leftrad} shows that contrary to
Lie algebras, where the non-degeneracy of the Killing form implies semisimplicity, the
converse of Proposition \ref{mindeg} is not true.
\vspace{.2cm}

Next, we give some insight into the structure of a finite-dimensional semisimple Leibniz algebra
in characteristic zero that can be obtained from the analogue of Levi's theorem for Leibniz
algebras and is due to Pirashvili \cite[Proposition 2.4]{P} and Barnes \cite[Theorem 1]{B3}.
The first part of Theorem \ref{strucsemisim} was already observed by Fialowski and Mih\'alka
\cite[Corollary 2.14]{FM} and the third part is \cite[Theorem~3.1]{GVKO}.

\begin{thm}\label{strucsemisim}
If $\lf$ is a finite-dimensional semisimple left Leibniz algebra over a field of characteristic zero,
then there exists a semisimple Lie subalgebra $\ssf$ of $\lf$ such that $\lf=\ssf\oplus\leib(\lf)$
and $\leib(\lf)$ is an anti-symmetric completely reducible $\ssf$-bimodule.\footnote{More precisely,
$\lf$ is the hemi-semidirect product of $\ssf$ and the $\ssf$-bimodule $\leib(\lf)$.} Moreover, if
$\lf$ is simple, then $\ssf$ is simple and $\leib(\lf)$ is irreducible.
\end{thm}

\begin{proof}
The existence of the semisimple Lie subalgebra $\ssf$ with $\lf=\ssf\oplus\leib(\lf)$ is an immediate
consequence of Levi's theorem for Leibniz algebras and Proposition~\ref{semisimrad}. It follows
from Proposition \ref{leftker} that $\leib(\lf)$ is an ideal of $\lf$ such that $\leib(\lf)\lf=0$. Hence
$\leib(\lf)$ is an anti-symmetric $\lf$-bimodule. In particular, $\leib(\lf)$ is also an anti-symmetric
$\ssf$-bimodule. Finally, Weyl's theorem (see \cite[Theorem 6.3]{H}) yields that $\leib(\lf)$ is a
completely reducible $\ssf$-bimodule.

Suppose now that $\lf$ is simple. Then we obtain from Proposition \ref{sim} that $\ssf\cong\lf
/\leib(\lf)=\lf_\lie$ is also simple. Let $M$ be a non-zero proper $\lf$-subbimodule of $\leib(\lf)$.
Then $M$ is an ideal of $\lf$ that is different from $0$, $\leib(\lf)$, and $\lf$, which contradicts
the simplicity of $\lf$. Hence $\leib(\lf)$ is an irreducible $\lf$-bimodule. By virtue of Proposition
\ref{ker}, $\leib(\lf)$ is also an irreducible $\ssf$-bimodule.
\end{proof}

A left or right Leibniz algebra $\lf$ is called {\em Lie-simple\/} if $\lf_\lie$ is simple. According to
Proposition \ref{simsemisim} and Proposition \ref{semisim}, Lie-simple left or right Leibniz algebras
are semisimple, but they are not always simple.
\vspace{.2cm}

\noindent {\bf Example} (see also \cite[Example 5.3]{DMS1}){\bf .} Let $\ssf$ be any simple
Lie algebra, and let $M$ and $M^\prime$ be two non-trivial irreducible left $\ssf$-modules.
Then consider the left Leibniz algebra $\lf:=\ssf\times[M\oplus M^\prime]$ with multiplication
$(x,a)(y,b):=(xy,x\cdot b)$ for any $x,y\in\ssf$ and any $a,b\in M\oplus M^\prime$ (see also
Example 3 in Section 2). The identity $(x,a)(x,a)=(0,x\cdot a)$ for any $x\in\ssf$ and any
$a\in M\oplus M^\prime$ shows that $\leib(\lf)\subseteq M\oplus M^\prime$. Since $M$ and
$M^\prime$ are non-trivial irreducible left $\ssf$-modules, we have that $\ssf M=M$ and
$\ssf M^\prime=M^\prime$. This in conjunction with the above identity shows the other
inclusion, and therefore $\leib(\lf)=M\oplus M^\prime$. Hence $\lf_\lie=\lf/\leib(\lf)\cong
\ssf$ is simple, and thus $\lf$ is Lie-simple. But as $M$ and $M^\prime$ are ideals of $\lf$
that are different from $0$, $\leib(\lf)$, and $\lf$, we conclude that $\lf$ is not simple.
\vspace{.2cm}

\noindent {\bf Remark.} The same argument as in the previous example in conjunction with
Proposition~\ref{semisim} proves that $\lf=\ssf\times M$ with multiplication $(x,a)(y,b):=
(xy,x\cdot b)$ for any $x,y\in\ssf$ and any $a,b\in M$, where $\ssf$ is a semisimple Lie
algebra and $M$ is a direct sum of non-trivial irreducible left $\ssf$-modules, is a semisimple
left Leibniz algebra. Note also that there exist semisimple left Leibniz algebras that cannot
be decomposed as a direct product of (Lie-)simple left Leibniz algebras (see \cite[Examples
2 and 3]{GVKO}).
\vspace{.2cm}

We obtain as an immediate consequence of Proposition \ref{semisim} in conjunction with
Corollary \ref{asoutder}\,(a) and the first Whitehead lemma for Lie algebras (see \cite[Theorem
13 in Chapter III]{J}) the corresponding result for symmetric bimodules over semisimple
Leibniz algebras.

\begin{thm}\label{1stwhiteheadlem}
If $\lf$ is a finite-dimensional semisimple left Leibniz algebra over a field of characteristic zero,
then $\HL^1(\lf,M)=0$ holds for every symmetric finite-dimensional $\lf$-bimodule $M$.
\end{thm}

\noindent {\bf Remark.}
Note that it follows from Corollary \ref {asadj} that Theorem \ref{1stwhiteheadlem} does not
hold for anti-symmetric bimodules over semisimple Leibniz algebras.
\vspace{.2cm}


\noindent {\bf Acknowledgments.} The author would like to thank Friedrich Wagemann for
reading a preliminary version of this paper and for making many very valuable suggestions.
The author is also grateful to Bakhrom Omirov for finding a mistake in the proof of the second
Whitehead lemma for Leibniz algebras which was contained in a previous version of this paper.




\begin{thebibliography}{99}


\bibitem{AAO1}
S. A. Albeverio, S. A. Ayupov, and B. A. Omirov:
On nilpotent and simple Leibniz algebras,
{\it Comm. Algebra\/} {\bf 33} (2005), no. 1, 159--172.

\bibitem{AAO2}
S. A. Albeverio, S. A. Ayupov, and B. A. Omirov:
Cartan subalgebras, weight spaces, and criterion of solvability of finite dimensional Leibniz algebras,
{\it Rev. Mat. Complut.\/} {\bf 19} (2006), no. 1, 183--195.

\bibitem{AOR}
S. Albeverio, B. A. Omirov, and I. S. Rakhimov:
Classification of 4-dimensional nilpotent complex Leibniz algebras,
{\it Extracta Math.\/} {\bf 21} (2006), no. 3, 197--210.

\bibitem{AO1}
S. A. Ayupov and B. A. Omirov:
On Leibniz algebras,
in: {\it Algebra and Operator Theory, Tashkent, 1997\/}
(eds. Yu. B. Khakimdjanov, M. Goze, and S. A. Ayupov),
Kluwer Acad. Publ., Dordrecht, 1998, pp. 1--12.

\bibitem{AO2}
S. A. Ayupov and B. A. Omirov:
On 3-dimensional Leibniz algebras,
{\it Uzbek. Mat. Zh.\/}, No. {\bf 1} (1999), 9--14.

\bibitem{B1}
D. W. Barnes:
Some theorems on Leibniz algebras,
{\it Comm. Algebra\/} {\bf 39} (2011), no. 7, 2463--2472.

\bibitem{B2}
D. W. Barnes:
On Engel's theorem for Leibniz algebras,
{\it Comm. Algebra\/} {\bf 40} (2012), no. 4, 1388--1389.

\bibitem{B3}
D. W. Barnes:
On Levi's theorem for Leibniz algebras,
{\it Bull. Austral. Math. Soc.\/} {\bf 86} (2012), no.~2, 184--185.

\bibitem{B}
T. Beaudouin:
{\it Etude de la cohomologie d'alg\`ebres de Leibniz via des suites spectrales\/},
Th\`ese de Doctorat, Universit\'e de Nantes, 2017.

\bibitem{BH}
S. Benayadi and S. Hidri:
Quadratic Leibniz algebras,
{\it J. Lie Theory\/} {\bf 24} (2014), no.~3, 737--759.

\bibitem{Bl1}
A. Bloh:
On a generalization of the concept of Lie algebra (in Russian),
{\it Dokl. Akad. Nauk SSSR\/} {\bf 165} (1965), 471--473;
translated into English in {\it Soviet Math. Dokl.\/} {\bf 6} (1965), 1450--1452.

\bibitem{Bl2}
A. Bloh:
Cartan-Eilenberg homology theory for a generalized class of Lie algebras (in Russian),
{\it Dokl. Akad. Nauk SSSR\/} {\bf 175} (1967), 266--268;
translated into English in {\it Soviet Math. Dokl.\/} {\bf 8} (1967), 824--826.

\bibitem{Bl3}
A. Bloh:
A certain generalization of the concept of Lie algebra (in Russian),
{\it Moskov. Gos. Ped. Inst. U\v{c}en. Zap.\/}, No. {\bf 375} (1971), 9--20.

\bibitem{Bou}
N. Bourbaki:
{\it Lie Groups and Lie Algebras: Chapters 1--3\/} (2$^{\rm nd}$ printing),
Springer-Verlag, Berlin/Heidelberg/New York/London/Paris/Tokyo, 1989.

\bibitem{CCO}
A. J. Calder\'on, L. M. Camacho, and B. A. Omirov:
Leibniz algebras of Heisenberg type,
{\it J. Algebra\/} {\bf 452} (2016), 427--447.

\bibitem{CGVO}
L. M. Camacho, S. G\'omez-Vidal, and B. A. Omirov:
Leibniz algebras associated to extensions of $\slf_2$,
{\it Comm. Algebra\/} {\bf 43} (2015), no. 10, 4403--4414.

\bibitem{CGVOK}
L. M. Camacho, S. G\'omez-Vidal, B. A. Omirov, and I. A. Karimjanov:
Leibniz algebras whose semisimple part is related to $\slf_2$,
{\it Bull. Malays. Math. Sci. Soc.\/} {\bf 40} (2017), no. 2, 599--615.

\bibitem{CK}
E. M. Ca\~nete and A Kh. Khudoyberdiyev:
The classification of 4-dimensional Leibniz algebras,
{\it Linear Algebra Appl.\/} {\bf 439} (2013), no. 1, 273--288.

\bibitem{Co}
S. Covez:
{\it The local integration of Leibniz algebras\/},
Th\`ese de Doctorat, Universit\'e de Nantes, 2010.

\bibitem{C}
C. Cuvier:
Alg\`ebres de Leibniz: d\'efinitions, propri\'et\'es,
{\it Ann. Sci. \'Ecole Norm. Sup.\/} (4) {\bf 27} (1994), no. 1, 1--45.

\bibitem{D}
I. Demir:
{\it Classification of 5-dimensional complex nilpotent Leibniz algebras\/},
Ph.D. Thesis, North Carolina State University, 2016.

\bibitem{DMS1}
I. Demir, K. C. Misra, and E. Stitzinger:
On some structures of Leibniz algebras,
in: {\it Recent Advances in Representation Theory, Quantum Groups, Algebraic Geometry,
and Related Topics\/} (eds. P. N. Achar, D. Jakeli\'c, K. C. Misra, and M. Yakimov), 
Contemp. Math., vol.~{\bf 623}, Amer. Math. Soc., Providence, RI, 2014, pp. 41--54.

\bibitem{DMS2}
I. Demir, K. C. Misra, and E. Stitzinger:
Classification of some solvable Leibniz algebras,
{\it Algebr. Represent. Theory\/} {\bf 19} (2016), no. 2, 405--417.

\bibitem{DMS3}
I. Demir, K. C. Misra, and E. Stitzinger:
On classification of four-dimensional nilpotent Leibniz algebras,
{\it Comm. Algebra\/} {\bf 45} (2017), no. 3, 1012--1018.

\bibitem{DMS4}
I. Demir, K. C. Misra, and E. Stitzinger:
Classification of $5$-dimensional complex nilpotent Leibniz algebras,
Preprint (55 pages), arXiv:1706.00951v1 [math.RA], June 3, 2017.

\bibitem{DW}
B. Dherin and F. Wagemann:
Deformation quantization of Leibniz algebras,
{\it Adv. Math.\/} {\bf 270} (2015), 21--48.

\bibitem{DA}
A. S. Dzhumadil'daev and S. A. Abdykassymova:
Leibniz algebras in characteristic $p$,
{\it C. R. Acad. Sci. Paris S\'er. I Math.\/} {\bf 332} (2001), no. 12, 1047--1052.

\bibitem{E}
S. Eilenberg:
Extensions of general algebras,
{\it Ann. Soc. Polon. Math.\/} {\bf 21} (1948), 125--134.

\bibitem{FM}
A. Fialowski and \'E. Zs. Mih\'alka:
Representations of Leibniz algebras,
{\it Algebr. Represent. Theory\/} {\bf 18} (2015), no. 2, 477--490.

\bibitem{GVKO}
S. G\'omez-Vidal, A. Kh. Khudoyberdiyev, and B. A. Omirov:
Some remarks on semisimple Leibniz algebras,
{\it J. Algebra\/} {\bf 410} (2014), 526--540.

\bibitem{G}
V. V. Gorbatsevich:
On the Lieification of Leibniz algebras and its applications,
{\it Russian Math.\/} ({\it Iz. VUZ\/}) {\bf 60} (2016), no. 4, 10--16.\footnote{A version
of this paper is also available on the arXiv as ``Some basic properties of Leibniz algebras"
(arXiv:1302.3345v2 [math.RA], March 1, 2013).}

\bibitem{H}
J. E. Humphreys:
{\it Introduction to Lie Algebras and Representation Theory\/} (Second printing, revised),
Graduate Texts in Mathematics, vol. {\bf 9},
Springer-Verlag, New York/Berlin, 1978.

\bibitem{J}
N. Jacobson:
{\it Lie Algebras\/},
Interscience Tracts in Pure and Applied Mathematics, No. {\bf 10},
Interscience Publishers (a division of John Wiley \& Sons), New York/London, 1962.

\bibitem{KRSH}
A. Kh. Khudoyberdiyev, I. S. Rakhimov, and Sh. K. Said Husain:
On classification of 5-dimensional solvable Leibniz algebras,
{\it Linear Algebra Appl.\/} {\bf 457} (2014), 428--454.

\bibitem{L}
J.-L. Loday:
Une version non commutative des alg\`ebres de Lie: les alg\`ebres de Leibniz,
{\it Enseign. Math.\/} (2) {\bf 39} (1993), no. 3-4, 269--293.

\bibitem{LP1}
J.-L. Loday and T. Pirashvili:
Universal enveloping algebras of Leibniz algebras and (co)homology,
{\it Math. Ann.\/} {\bf 296} (1993), no. 1, 139--158.

\bibitem{LP2}
J.-L. Loday and T. Pirashvili:
Leibniz representations of Lie algebras,
{\it J. Algebra\/} {\bf 181} (1996), no. 2, 414--425.

\bibitem{MY}
G. Mason and G. Yamskulna:
Leibniz algebras and Lie algebras,
{\it SIGMA Symmetry Integrability Geom. Methods Appl.\/} {\bf 9} (2013), Paper 063, 10 pp.

\bibitem{ORT}
B. A. Omirov, I. S. Rakhimov, and R. M. Turdibaev:
On description of Leibniz algebras corresponding to $\slf_2$,
{\it Algebr. Represent. Theory\/} {\bf 16} (2013), no. 5, 1507--1519.

\bibitem{OW}
B. Omirov and F. Wagemann:
Rigidity and cohomology of Leibniz algebras,
Preprint (20 pages), arXiv:1508.06877v4 [math.KT], October 6, 2015.

\bibitem{P1}
A. Patsourakos:
On nilpotent properties of Leibniz algebras,
{\it Comm. Algebra\/} {\bf 35} (2007), no.~12, 3828--3834.

\bibitem{P2}
A. Patsourakos:
On solvable Leibniz algebras,
{\it Bull. Greek Math. Soc.\/} {\bf 55} (2008), 59--63.

\bibitem{P}
T. Pirashvili: On Leibniz homology,
{\it Ann. Inst. Fourier\/} (Grenoble) {\bf 44} (1994), no. 2, 401--411.

\bibitem{S}
R. D. Schafer:
{\it An Introduction to Nonassociative Algebras\/},
Pure and Applied Mathematics, A Series of Monographs and Textbooks, vol. {\bf 22},
Academic Press, Inc., New York, 1966.

\bibitem{Se}
G. B. Seligman:
{\it Modular Lie Algebras\/},
Ergebnisse der Mathematik und ihrer Grenzgebiete, Band {\bf 40},
Springer-Verlag New York, Inc., New York, 1967.

\bibitem{SF}  
H. Strade and R. Farnsteiner: 
{\it Modular Lie Algebras and Their Representations\/}, 
Monographs and Textbooks in Pure and Applied Mathematics, vol. {\bf 116}, 
Marcel Dekker, New York/Basel, 1988.


\end{thebibliography}
\end{document}